\documentclass[10pt, leqno]{amsart}

\usepackage{amsmath,amssymb,amsthm, epsfig}
\usepackage[mathscr]{eucal}
\usepackage{mathptmx}
\usepackage{hyperref}
\usepackage{stackrel}

\def \diam {\mathrm{diam}}

\usepackage{epsfig}

\title{Fully nonlinear singularly perturbed models with non-homogeneous degeneracy}
\author{Jo\~{a}o Vitor da Silva, \quad Elzon C. J\'{u}nior \quad and \quad Gleydson C. Ricarte}

\def \R {\mathbb{R}}
\def \supp {\mathrm{supp } }
\def \Div {\mathrm{div}}
\def \dist {\mathrm{dist}}
\def \diam {\mathrm{diam}}

\def \loc {\mathrm{loc}}

\def \suchthat {\ \big | \ }

\def \tr {\mathrm{tr}}

\def \Leb {\mathscr{L}^N}
\def \e {\epsilon}
\newcommand{\defeq}{\mathrel{\mathop:}=}

\newtheorem{theorem}{Theorem}[section]
\newtheorem{lemma}[theorem]{Lemma}
\newtheorem{proposition}[theorem]{Proposition}
\newtheorem{corollary}[theorem]{Corollary}

\theoremstyle{definition}

\newtheorem{definition}[theorem]{Definition}
\newtheorem{example}[theorem]{Example}

\theoremstyle{remark}

\newtheorem{remark}[theorem]{Remark}

\numberwithin{equation}{section}

\newcommand{\pe}{P_{\epsilon}}
\newcommand{\intav}[1]{\mathchoice {\mathop{\vrule width 6pt height 3 pt depth  -2.5pt
\kern -8pt \intop}\nolimits_{\kern -6pt#1}} {\mathop{\vrule width
5pt height 3  pt depth -2.6pt \kern -6pt \intop}\nolimits_{#1}}
{\mathop{\vrule width 5pt height 3 pt depth -2.6pt \kern -6pt
\intop}\nolimits_{#1}} {\mathop{\vrule width 5pt height 3 pt depth
-2.6pt \kern -6pt \intop}\nolimits_{#1}}}

\begin{document}

\begin{abstract}
This work is devoted to studying non-variational, nonlinear singularly perturbed elliptic models enjoying a double degeneracy character with prescribed boundary value in a domain. In its simplest form, for each $\varepsilon>0$ fixed, we seek a non-negative function $u^{\epsilon}$ satisfying
$$
\left\{
\begin{array}{rclcl}
\left[|\nabla u^{\varepsilon}|^p + \mathfrak{a}(x)|\nabla u^{\varepsilon}|^q \right] \Delta u^{\varepsilon} & = & \zeta_{\varepsilon}(x, u^{\varepsilon}) & \mbox{in} & \Omega,\\
u^{\varepsilon}(x) & = & g(x) & \mbox{on} & \partial \Omega,
\end{array}
\right.
$$
in the viscosity sense for suitable data $p, q \in (0, \infty)$, $\mathfrak{a}$, $g$, where $\zeta_{\varepsilon}$ one behaves singularly of order $\mbox{O} \left(\e^{-1} \right)$ near $\epsilon$-level surfaces. In such a context, we establish the existence of certain solutions. We also prove that solutions are locally (uniformly) Lipschitz continuous, and they grow in a linear fashion. Moreover, solutions and their free boundaries possess a sort of measure-theoretic and weak geometric properties. Particularly, for a restricted class of non-linearities, we prove the finiteness of the $(N-1)$-dimensional Hausdorff measure of level sets. At the end, we address a complete and in-deep analysis concerning the asymptotic limit as $\varepsilon \to 0^{+}$, which is related to one-phase solutions of inhomogeneous non-linear free boundary problems in flame propagation and combustion theory. Finally, we also present some fundamental regularity tools in the theory of doubly degenerate fully nonlinear elliptic PDEs, which may have their own mathematical interest.

\medskip


\medskip

\end{abstract}

\maketitle


\section{Introduction}

In this manuscript we shall develop an approach to study (locally) sharp and geometric estimates of one-phase solutions to singularly perturbed problems having a non-homogeneous double degeneracy, whose mathematical model is given by: Fixed a parameter $\varepsilon \in (0, 1)$, we would like to find
\begin{equation}\label{Equation Pe}\tag{$\pe$}
u^{\epsilon} \ge 0\,\,\,\text{viscosity solution to}\,\,\,
\left\{
	\begin{array}{rclcl}
		\mathcal{H}(x, \nabla u^{\epsilon})F(x, D^2 u^{\epsilon}) & = &  \zeta_{\epsilon}(x, u^{\epsilon})   & \mbox{in} & \Omega\\
u^{\varepsilon}(x) & = & g(x) & \mbox{on} & \partial \Omega,
	\end{array}
\right.
\end{equation}
for a bounded and open set $\Omega \subset \R^N$, where $0\le g \in C^0(\partial \Omega)$, and $F$ is a second order, fully non-linear (uniformly elliptic) operator, i.e., non-linear in its highest derivatives.

We will focus our attention to reaction-diffusion models with singular behavior of order $\mbox{O} \left( \e^{-1} \right)$ near $\epsilon$-\verb"level layers", i.e. $\{u_\varepsilon\sim \varepsilon\}$. Furthermore, the diffusion process is assumed to be anisotropic and doubly degenerate, thereby collapsing as $|\nabla u^\epsilon| \sim 0$.

In a few words, under the appropriated hypothesis on data, we show that, for $\varepsilon \to 0^{+}$, the family of solutions $\{u^{\varepsilon}\}_{\varepsilon>0}$ to \eqref{Equation Pe} are asymptotic approximations to a one-phase solution $u_0$ of an inhomogeneous non-linear free boundary problem (for short FBP), which arises in the mathematical formulation of some issues in flame propagation and combustion theory (stationary setting - cf. \cite{CK-AM04}, \cite{LVW01} and \cite{Weiss03}).

\subsection{Main assumptions}

From here we will assume the following structural assumptions:
\begin{itemize}
  \item[(A0)]({{\bf Continuity and normalization condition}})
  $$
  \text{Fixed}\,\, \Omega \ni x \mapsto F(x, \cdot) \in C^{0}(\text{Sym}(N)) \quad  \text{and} \quad F(\cdot, \text{O}_{N\times N}) = 0.
  $$
  \item[(A1)]({{\bf Uniform ellipticity}}) For any pair of matrices $\mathrm{X}, \mathrm{Y}\in Sym(N)$
$$
    \mathscr{M}^{-}_{\lambda,\Lambda}(\mathrm{X}-\mathrm{Y})\leq F(x, \mathrm{X})-F(x, \mathrm{Y})\leq \mathscr{M}^{+}_{\lambda,\Lambda}(\mathrm{X}-\mathrm{Y})
$$
where $\mathscr{M}^{\pm}_{\lambda,\Lambda}$ stand for \textit{Pucci's extremal operators} given by
\[
   \mathscr{M}_{\lambda, \Lambda}^-(\mathrm{X})\defeq \lambda\sum_{e_i>0}e_i(\mathrm{X})+\Lambda\sum_{e_i<0}e_i(\mathrm{X})\quad\textrm{ and }\quad \mathscr{M}_{\lambda, \Lambda}^+(X)\defeq \Lambda\sum_{e_i>0}e_i(\mathrm{X})+\lambda\sum_{e_i<0}e_i(\mathrm{X})
\]
for \textit{ellipticity constants} $0<\lambda\leq \Lambda< \infty$, where $\{e_i(\mathrm{X})\}_{i}$ are the eigenvalues of $\mathrm{X}$.

Moreover, for our Lipschitz estimates, we must require some sort of continuity assumption on coefficients:

  \item[(A2)]({{\bf $\omega-$continuity of coefficients}}) There exist a uniform modulus of continuity $\omega: [0, \infty) \to [0, \infty)$ and a constant $\mathrm{C}_{\mathrm{F}}>0$ such that
$$
\Omega \ni x, x_0 \mapsto \Theta_{\mathrm{F}}(x, x_0) \defeq \sup_{\substack{\mathrm{X} \in  Sym(N) \\ \mathrm{X} \neq 0}}\frac{|F(x, \mathrm{X})-F(x_0, \mathrm{X})|}{\|\mathrm{X}\|}\leq \mathrm{C}_{\mathrm{F}}\omega(|x-x_0|),
$$
which measures the oscillation of coefficients of $F$ around $x_0$. Finally, we define
$$
\|F\|_{C^{\omega}(\Omega)} \defeq \inf\left\{\mathrm{C}_{\mathrm{F}}>0: \frac{\Theta_{\mathrm{F}}(x, x_0)}{\omega(|x-x_0|)} \leq \mathrm{C}_{\mathrm{F}}, \,\,\, \forall \,\,x, x_0 \in \Omega, \,\,x \neq x_0\right\}.
$$
\end{itemize}

In our studies, the diffusion properties of the model \eqref{Equation Pe} degenerate along an \textit{a priori} unknown set of singular points of existing solutions:
$$
   \mathscr{S}_0(u, \Omega^{\prime}) \defeq \{x \in \Omega^{\prime} \Subset \Omega: |\nabla u(x)| = 0\}.
$$
For this reason, we will enforce that $\mathcal{H}:\Omega \times \R^N \to [0, \infty)$ one behaves as
\begin{equation}\label{1.2}
     L_1 \cdot \mathcal{K}_{p, q, \mathfrak{a}}(x, |\xi|) \leq \mathcal{H}(x, \xi)\leq L_2 \cdot \mathcal{K}_{p, q, \mathfrak{a}}(x, |\xi|)
\end{equation}
for constants $0<L_1\le L_2< \infty$, where
\begin{equation}\label{N-HDeg}\tag{\bf N-HDeg}
   \mathcal{K}_{p, q, \mathfrak{a}}(x, |\xi|) \defeq |\xi|^p+\mathfrak{a}(x)|\xi|^q, \,\,\,\text{for}\,\,\, (x, \xi) \in \Omega \times \R^N.
\end{equation}

In addition, for the non-homogeneous degeneracy \eqref{N-HDeg}, we suppose that the exponents $p, q$ and the modulating function $\mathfrak{a}(\cdot)$ fulfil
\begin{equation}\label{1.3}
   0< p \le q< \infty \qquad \text{and} \qquad \mathfrak{a} \in C^0(\Omega, [0, \infty)).
\end{equation}
Roughly speaking, $\mathcal{H}$ satisfies distinct growths at origin and at infinity:
$$
   \displaystyle \lim_{|\xi| \to 0^{+}} \frac{\mathcal{H}(x, \xi)}{|\xi|^p(1+\mathfrak{a}(x))} \in (0, \infty) \quad \text{and} \quad \lim_{|\xi| \to +\infty} \frac{\mathcal{H}(x, \xi)}{|\xi|^q(1+\mathfrak{a}(x))} \in [0, \infty) \,\,\, (\text{uniformly in}\,\, x \in \Omega).
$$
Moreover, (for $p \neq q$)
$$
    \displaystyle \lim_{|\xi| \to 0^{+}} \frac{\mathcal{H}(x, \xi)}{|\xi|^q} = +\infty
\quad \text{and} \quad \displaystyle \lim_{|\xi| \to +\infty} \frac{\mathcal{H}(x_0, \xi)}{|\xi|^p} =
\left\{
\begin{array}{lcl}
	\verb"finite" & \text{if} & \mathfrak{a}(x_0)=0 \\
+\infty & \text{if} & \mathfrak{a}(x_0)>0
	\end{array}
\right.
$$

In turn, in our research, the reaction term, i.e. $\zeta_\epsilon \colon \Omega \times \mathbb{R}_{+} \to \mathbb{R}_{+}$, represents the singular perturbation of the model. In this point, we are interested in a singular behaviour of order $\mbox{O} \left ( \frac{1}{\e} \right )$ along $\e$-level layers $\{u_\epsilon\sim \epsilon\}$. Hence, we are led to consider reaction terms fulfilling
 \begin{equation}\label{Cond1 zeta}
	 \mathcal{B}_0 \le \zeta_{\epsilon}(x, t) \le \frac{\mathcal{A}}{\epsilon} \chi_{(0,\epsilon)}(t) + \mathcal{B}, \quad \forall \,\,(x,t) \in \Omega \times \mathbb{R}_{+},
\end{equation}
for nonnegative constants $\mathcal{A}, \, \mathcal{B}_0, \,\mathcal{B} \ge 0$. Notice that $\zeta_\epsilon \equiv 0$ satisfies \eqref{Cond1 zeta}. Nevertheless, we shall also impose the following non-degeneracy assumption in order to ensure that such a reaction term enjoys an authentic singular character:
\begin{equation} \label{nondeg_RT}
	 \mathscr{I} \defeq \inf\limits_{\Omega \times [t_0, \mathrm{T}_0]}  \epsilon \zeta_\epsilon(x, \epsilon t) > 0,
\end{equation}
for some constants $0\le t_0< \mathrm{T}_0< \infty$, where $\mathscr{I}$ does not depend on $\epsilon$. Intuitively, \eqref{nondeg_RT} reads that the singular term behaves as $\sim \frac{1}{\e} \chi_{(0,\e)}$ plus a non-negative noise that stays uniformly controlled. Indeed, simpler cases covered by our analysis are singular reaction terms built up as a multiple of the approximation of unity plus a uniform bounded function
\begin{equation}\label{zeta}
	\zeta_{\epsilon}(x, t) \defeq \mathcal{Q}(x)\frac{1}{\epsilon} \zeta \left(\frac{t}{\epsilon}\right) + f_\epsilon(x).
\end{equation}
For such approximations, $0< \mathcal{Q} \in C^0(\overline{\Omega})$, $0\leq \zeta \in C^{\infty}(\R)$  with $\supp ~\zeta = [0,1]$, and $f_\epsilon$ is a non-negative continuous function bounded away from infinity. Finally, it is readily verifiable that the reaction term in \eqref{zeta} does fulfill \eqref{Cond1 zeta} and \eqref{nondeg_RT} with $\mathcal{A} = \|\mathcal{Q}\|_{L^{\infty}(\Omega)}\|\zeta\|_{L^{\infty}(\R_{+})}$, $\displaystyle \mathcal{B}_0 = \inf_{\Omega} f_{\e}(x)$ and $\mathcal{B} = \|f_{\e}\|_{L^{\infty}(\Omega)}$.

\subsection{Statement of main results}\label{DefPreRes}

\hspace{0.3cm} To formulate our main results, we need to introduce some definitions. We will start with the definition of the viscosity solution to
\begin{equation}\label{MainOperator}
  \mathcal{G}(x, \nabla u, D^2u) \defeq  \mathcal{H}(x, \nabla u)F(x, D^2 u).
\end{equation}

\begin{definition}[{\bf Viscosity solution}] A function $u \in C(\Omega)$ is called a viscosity sub-solution (super-solution) of
$$
\mathcal{G}(x, \nabla u(x), D^2u(x)) = f(x, u(x)) \quad \mbox{in} \quad \Omega,
$$
if whenever $\phi \in C^2(\Omega)$ and $u-\phi$ has a local maximum (minimum) at $x_0 \in \Omega$ there holds
$$
 \mathcal{G}(x, \nabla \phi(x_0), D^2 \phi(x_0)) \geq f(x_0, \phi(x_0)) \quad (\mbox{resp.} \leq f(x_0, \phi(x_0))).
$$
Finally, a function $u$ is a viscosity solution when it is a simultaneously a viscosity sub and super-solution.
\end{definition}

In order to prove some key geometric properties of solutions, it is pivotal to adopt a more appropriate notion of viscosity solution. As a matter of fact, \eqref{Equation Pe} has a lack of comparison principle, thus uniqueness assertions might not be true. Therefore, we shall make a particular election of solutions. For this reason, the \textit{least supersolution approach} takes place in our studies by way of \textit{Perron type solutions}.

\begin{definition}[{\bf Perron type solution}]
Throughout this manuscript we will work with Perron type solutions to the singularly perturbed problem \eqref{Equation Pe}. Precisely, fixed a viscosity sub-solution $u_\star$, and a viscosity super-solution  $u^\star$ to \eqref{Equation Pe}, fulling $u_\star \le u^\star,$ in $\Omega$, the Perron solution $u^\epsilon$ is given by
\begin{equation}\label{Perron0}
	u^\epsilon(x) = \inf \left \{ w(x) \suchthat w \text{ is a super-solution to } \eqref{Equation Pe}, \text{ and } u_\star \le w \le u^\star \right \}.
\end{equation}

It is worth noting that for each $\epsilon>0$ fixed, the existence of such a Perron's solution follows by sub/supersolutions methods, see \textit{e.g.} \cite[Theorem 4.1]{UserG}.  Therefore, from now on, by a solution $u^\epsilon$  to \eqref{Equation Pe}, we denote a Perron type solution built-up as in \eqref{Perron0}.
\end{definition}

We establish existence of Perron type solutions to \eqref{Equation Pe} in our first result.

\begin{theorem}[{\bf Existence of Perron solutions}]\label{ExistMinSol} Let $\Omega \subset \R^n$ be a bounded Lipschitz domain and let $0 \le g \in C(\partial \Omega)$ be a boundary datum. Then, for each fixed $\varepsilon>0$ there exists a non-negative viscosity solution $u^{\varepsilon} \in C(\overline{\Omega})$ to \eqref{Equation Pe}.
\end{theorem}

We prove uniform gradient estimates, which supplies local compactness in the uniform convergence topology in our next result.

\begin{theorem}[{\bf Optimal Lipschitz estimate}]\label{thmreg}
Let $\{u^{\epsilon}\}_{\epsilon >0}$ be a solution of \eqref{Equation Pe}. Given $\Omega^{\prime} \Subset \Omega$, there exists a constant $C_0$ depending on dimension, ellipticity constants and on $\Omega^\prime$, but independent of $\epsilon>0$, such that
$$
	\|\nabla u^{\epsilon}\|_{L^{\infty}(\Omega^\prime)} \le \mathrm{C}_0.
$$
Additionally, if $\{u^{\epsilon}\}_{\epsilon >0}$ is a uniformly bounded family\footnote{Such a bound will be universal, i.e., it will depend only on data of the problem. Moreover, this statement is obtained via the application of Alexandroff-Bakelman-Pucci estimate adapted to our context.}, then it is pre-compact in the Lipschitz topology.
\end{theorem}

We stress that a key ingredient in order to prove optimal Lipschitz regularity is the $C_{\text{loc}}^{1, \alpha}$ estimates addressed in \cite{daSR20}. Other important pieces of information are the versions of the Harnack inequality and inhomogeneous Hopf type result adapted to our double degenerate context (see, Appendices \ref{Append} for more details).

From now on, we will label the distance of a point in the non-coincidence set $x_0\in \Omega \cap \{u^\epsilon > 0\}$ to the approximating transition boundary, $\Gamma_\epsilon$, by
$$
d_{\epsilon}(x_0) \defeq \textrm{dist}(x_0, \{u^{\epsilon} \le \epsilon\}).
$$

Next, we prove that, inside $\{u^{\e}>\e\}$, solutions grow in a linear fashion away from $\e$-level surfaces.

\begin{theorem}[{\bf Linear growth}]\label{cresc}
Let $\{u^{\epsilon}\}_{\epsilon >0}$ be a Perron's solution of \eqref{Equation Pe}. There exists  $\mathrm{c}(\verb"universal parameters")>0$ such that, for $x_0 \in \{u^{\epsilon} > \epsilon\}$ and $0<\epsilon \ll d_\epsilon(x_0) \ll 1$, there holds
$$
 	u^{\epsilon}(x_0) \ge \mathrm{c} \cdot d_{\epsilon}(x_0).
$$
\end{theorem}

The proof of Linear growth consists of combining the construction of an appropriate barrier function with the minimality of Perron solutions. Such an instrumental idea was first introduced in the last author's works \cite{ART17} and \cite{RT11} for the fully nonlinear scenario.

In a free boundary point of view, it is important highlighting that viscosity solutions of \eqref{Equation Pe} develop two ``distinct free boundaries''. The first one is the set of singular points of existing solutions $\mathscr{S}_0(u^{\varepsilon}, \Omega^{\prime})$, and the second one is the so-named ``physical transition '', i.e. $\Gamma_{\varepsilon} = \{u^{\varepsilon} \thicksim \varepsilon \}$ ($\varepsilon-\text{level surfaces}$). One of most the difficult tasks in our research consists in showing that these two free boundaries do not intersect in measure. As a matter of fact, we are able to obtain a uniform lower/upper control of $u^{\varepsilon}$ in terms of $\dist(\cdot, \Gamma_{\varepsilon})$:
$$
\dist(x_0, \Gamma_{\varepsilon}) \lesssim u^{\varepsilon}(x_0) \lesssim \dist(x_0, \Gamma_{\varepsilon}).
$$

Next, we prove that Perron type solutions are strongly non-degenerate near $\e$-level surfaces. Summarily, the maximum of $u^{\e}$ on the boundary of a ball $B_r(x_0)$, centered in $\{u^{\e}>\e\}$, is of the order of $r$.

\begin{theorem}[{\bf Strong Non-degeneracy}]\label{degforte}
Given $\Omega^{\prime}\Subset \Omega$, there exists a positive constant $\mathrm{c}(\verb"universal")$ such that, for $x_0 \in \{u^{\epsilon} > \epsilon\}$, $\epsilon \ll \rho \ll 1$,  there holds
$$
	\mathrm{c}\cdot\rho \le \sup_{B_{\rho}(x_0)}u^{\epsilon}(x) \le \mathrm{c}^{-1}\cdot(\rho + u^{\epsilon}(x_0)).
$$
\end{theorem}

As a consequence of Theorem \ref{thmreg} we get the following result:

\begin{theorem}[{\bf The limiting PDE}] Let $u^{\varepsilon}$ be a solution to \eqref{Equation Pe}, then for any sequence $\varepsilon_k \to 0^{+}$ there exist a subsequence $\varepsilon_{k_j} \to 0^{+}$ and $u_0\in C^{0,1}_{\loc}(\Omega)$ such that
\begin{enumerate}
	\item[(1)] $u^{\varepsilon_{k_j}} \to u_0$ locally uniformly in $\Omega$;
	\item[(2)] $u_0 \in [0, K_0]$ in $\overline{\Omega}$ for some constant $K_0(\verb"universal")>0$ (independent of $\varepsilon$);
	\item[(3)] $ \mathcal{G}(x, \nabla u_0, D^2 u_0) = f_0(x)$ in $\{u_0>0\}$, with $0\le f_0 \in L^{\infty}(\Omega)\cap C^0(\Omega)$.
\end{enumerate}
\end{theorem}

Let us introduce the notation:
$$
   \mathfrak{F}(u_0, \mathcal{O}) \defeq \partial \{u_0>0\} \cap \mathcal{O}.
$$

\begin{theorem}[{\bf Asymptotic behavior close free boundary}]\label{limite} Let $\Omega^{\prime} \Subset \Omega$. Fix $x_0 \in \{u_0>0\} \cap \Omega^{\prime}$ such that $\dist(x_0, \mathfrak{F}(u_0, \Omega^{\prime}))\leq \frac{1}{2}\dist(\Omega^{\prime}, \partial \Omega)$. Then there exists a constant $\mathrm{C}>0$ independent of $\varepsilon$ such that
\begin{equation}\label{ocontrol}
	\mathrm{C}^{-1}\cdot \dist(x_0,\mathfrak{F}(u_0, \Omega^{\prime})) \le u_{0}(x_0) \le \mathrm{C}\cdot  \dist(x_0,\mathfrak{F}(u_0, \Omega^{\prime})).
\end{equation}
\end{theorem}

Finally, we will prove that the limiting free boundary $\mathfrak{F}(u_0, \Omega^{\prime})$ has local finite $\mathcal{H}^{N-1}$-Hausdorff measure. To this end, we must restrict our analysis to the class of operators satisfying an Asymptotic Concavity Property, which will be stated precisely in Section \ref{Sct HE}.

\begin{theorem}[{\bf Hausdorff estimates}]\label{limit3} Given $\Omega^{\prime} \Subset \Omega$, there exists a positive constant $\mathrm{C}(\Omega^{\prime}, \verb"universal parameters")$ such that, for $x_0 \in \mathfrak{F}(u_0, \Omega^{\prime})$,
$$
	\mathcal{H}^{N-1}(\mathfrak{F}(u_0, \Omega^{\prime}) \cap B_{\rho}(x_0)) \le \mathrm{C}\cdot\rho^{N-1}.
$$
Additionally, there exists a positive constant $\mathrm{C}_1(\Omega^{\prime}, \verb"universal parameters")$, such that for $\rho \ll 1$ and $x_0\in \mathfrak{F}(u_0, \Omega^{\prime})$, there holds
$$
\mathrm{C}_1^{-1}\cdot\rho^{N-1} \le \mathcal{H}^{N-1}(\mathfrak{F}_{\text{red}}(u_0, \Omega^{\prime}) \cap B_{\rho}(x_0)) \le \mathrm{C}_1\cdot\rho^{N-1}
$$
where $\mathfrak{F}_{\text{red}}(u_0, \Omega^{\prime})\defeq \partial_{\text{red}}\{u_0 >0\} \cap \Omega^{\prime}$ is the reduced transition boundary\footnote{The reduced free boundary, i.e. $\partial_{\text{red}}\{u_0 >0\}$ is a subset of $\partial\{u_0 >0\}$ where there exists, in the measure theoretic sense, the normal vector, see the Monograph \cite{EG92} for a survey concerning geometric measure theory.}. Particularly,
$$
\mathcal{H}^{N-1}(\mathfrak{F}(u_0, \Omega^{\prime})\setminus \mathfrak{F}_{\text{red}}(u_0, \Omega^{\prime}))=0.
$$
\end{theorem}

In conclusion, it is worth highlighting that our findings extend regarding non-variational scenario, former results from \cite{ART17}, \cite{MoWan1}, \cite{RT11} and \cite{T0}, and to some extent, those from \cite{DPS03}, \cite{dosPT16}, \cite{LW16}, \cite{MW09}, \cite{MorTei07} and \cite{T1}, concerning degenerate and variational models, by making using of different systematic approaches and techniques adjusted to the framework of fully non-linear models with non-homogeneous degeneracy. Moreover, they are new even for the toy model:
$$
 \left[|\nabla u^{\epsilon}|^p+\mathfrak{a}(x)| \nabla u^{\epsilon}|^q\right] \Delta u^{\epsilon}= \mathcal{Q}(x)\frac{1}{\epsilon}\zeta \left(\frac{u^{\epsilon}}{\epsilon}\right) + f_\epsilon(x) \,\, (\text{with \eqref{1.3} and \eqref{zeta} in force}).
$$

Lastly, it is noteworthy to point out that in order to establish our findings, we have developed pivotal auxiliary tools which, according to our scientific knowledge, were not available in the current literature for our model equations. Thereby they may have their own mathematical interest. Among these, we must quote, Weak and Harnack inequalities, Local Maximum principle, H\"{o}lder regularity, ABP estimate and inhomogeneous Hopf type results, just to mention a few (see, Appendices \ref{Append} for more details).

\subsection{Motivations and State-of-Art}

The mathematical theory of singular perturbation concerns a wide class of methods employed in several fields of Mathematics, Physics and their affine areas (see, \cite{H13} for an introductory essay). As a matter of fact, penalization methods were pivotal in studying certain discontinuous minimization problems in the theory of critical points of non-differentiable functionals, where the Alt-Caffarelli's seminal work \cite{AC81} marks the genesis of such a theory by carrying out the analysis of the minimization problem
\begin{equation}\label{EqAltCaff}
   \displaystyle \stackrel[H_0^1(\Omega)\atop{v_{|\partial \Omega}=g}]{}{\text{min}} \int_{\Omega} \left(\frac{1}{2}|\nabla v|^2 + \mathcal{Q}(x)\chi_{\{v>0\}}\right)dx \quad \text{for suitable data} \,\,g\ge 0 \,\,\text{and} \,\,\mathcal{Q}>0.
\end{equation}
Historically, such a variational problem \eqref{EqAltCaff} has appeared in the mathematical formulation of a variety of relevant one-phase models: cavity type problems \cite{dosPT16}, jets problem (see, \cite[Ch. 1]{CS05} and therein references), optimal design problems \cite[Chapter 6]{Tei-Book}, just to mention few of them. Notice that the Euler-Lagrange equation to \eqref{EqAltCaff} is
$$
   \Delta u_0(x) = \mathcal{Q}(x)\delta_0(u_0) \quad \text{in} \quad \Omega
$$
in an appropriate distributional sense, addressed in \cite{AC81}. Thereby, minimizers to \eqref{EqAltCaff} are obtained as the uniform limit when $\e \to 0^{+}$ of the problem
$$
   \Delta u^{\varepsilon}(x) = \mathcal{Q}^2(x)\beta_{\varepsilon}(u^{\varepsilon}) \quad \text{in} \quad \Omega \quad (\text{for} \,\, \beta_{\varepsilon} \sim \e^{-1}\chi_{(0,\varepsilon)}).
$$

Therefore, the core idea of studying approximation solutions is that small perturbations for certain elliptic problems propagate in a quantifiable fashion. Thus, analysing perturbed solutions can be useful to establish regularity estimates for the desired minimal solution of \eqref{EqAltCaff} and its free boundary.

Such an influential idea can also be employed in analysing over-determined problems as follows: given $\Omega \subset \mathbb{R}^N$ a bounded and smooth domain and functions $0 \leq f, g \in C(\overline{\Omega})$ and $0< \mathcal{Q} \in C^0(\overline{\Omega})$, we would like to find a ``compact hyper-surface'' $\Gamma_0 \defeq \partial \Omega^{\prime} \subset \Omega$ such that the inhomogeneous \textit{one-phase Bernoulli-type problem}
\begin{equation}\label{FreeBernProb}
\left\{
\begin{array}{rcccl}
\mathcal{L}\,u(x)  &=&  f(x) & \mbox{in} & \Omega \backslash \Omega^{\prime}\\
u(x) & = & g(x) & \mbox{on} & \partial \Omega\\
u(x) & = & 0 &\mbox{on} & \Omega^{\prime}\\
\frac{\partial u}{\partial \nu}(x) & = & \mathcal{Q}(x) & \mbox{on} & \Gamma_0 \quad (\text{in a suitable sense})
\end{array}
\right.
\end{equation}
admits a non-negative solution for a second order elliptic operator $\mathcal{L}$ (in divergence or in non-divergence form) with suitable structure. As above, limiting solutions coming from certain approximating regularized problems are natural profiles to solve \eqref{FreeBernProb} (in an appropriate sense with $\Gamma_0 \defeq \partial \{u>0\}$). This will motivate the next paragraph.

\subsubsection{Modern developments in singular perturbation theory}
Our impetus for current investigations in this work also comes from their intrinsic connections with non-linear one-phase problems, which arise in the mathematical theory of combustion, as well as in the study of flame propagation problems (stationary setting). Precisely, they appear in the description of laminar flames as an asymptotic limit for the non-linear formulation of high energy activation models with source terms (cf. \cite{CK-AM04}, \cite{DPS03}, \cite{Kar19}, \cite{LVW01} and \cite{Weiss03}). In a general framework, such models corresponds to the limit as $\e \to 0$ in \eqref{Equation Pe}, i.e. a one-phase inhomogeneous FBP, where the reaction-diffusion is driven by a doubly degenerate operator (cf. \cite{ART17}, \cite{MoWan1} and \cite{RT11}):
{\small{
\begin{equation}\label{EqFBP}\tag{I-FBP-NH}
\left\{
\begin{array}{rcccl}
 \mathcal{G}(x, \nabla u, D^2 u)& = & f(x) & \mbox{in} & \{u>0\}, \quad (\text{for}\,\,\,f \in C^0(\Omega)\cap L^{\infty}(\Omega))\\
u(x) & \ge & 0 &\mbox{in} & \Omega\\
\mathrm{H}(x, |\nabla u(x)|) & \le & \mathcal{T}(x) & \mbox{on} & \partial \{u>0\}, \quad (\text{for}\,\, 0<\mathcal{T}\in C^{0}(\overline{\Omega}))\\
u(x) & = & g(x) & \mbox{on} & \partial \Omega.
\end{array}
\right.
\end{equation}}}
The condition that $\mathrm{H}$ enforces on $u$ is commonly referred to as \textit{free boundary condition}.

The mathematical development of theses regularized problems has yielded important scientific breakthroughs in the free boundary theory. Historically, their studies date back to Berestycki-Caffarelli-Nirenberg's pioneering work \cite{BCN}, where the linear elliptic scenario was addressed (cf. \cite{T1} for the analysis of elliptic PDEs of the flame propagation type via a variational treatment)
$$
\displaystyle   \mathcal{L}[u] \defeq \sum_{i, j =1}^{N} a_{ij}(x)D_{ij} u(x) + \sum_{i=1}^{N} b_i(x)D_i u(x) + c(x)u(x) = \beta_{\varepsilon}(u), \quad (\text{for}\,\,\,C^1 \text{coefficients}).
$$

Before presenting the recent progresses in the fully non-linear scenario, we must quote some fundamental contributions of several authors regarding homogeneous/inhomogeneous singular perturbation problems (one and two-phases and their parabolic counterpart), as well as variational problems with uniformly elliptic and degenerate structure, see \cite{CLW1}, \cite{CLW2}, \cite{DPS03}, \cite{Kar19}, \cite{LW98}, \cite{LW06}, \cite{LW16}, \cite{MorTei07}, \cite{MW09} and \cite{MoWan1} for an extensive but incomplete list of such investigations:
{\small{
$$
\mathcal{L} u^{\varepsilon}(x) \defeq
\left\{
\begin{array}{rcll}
\Div(A(x)\nabla u^{\varepsilon})  &=&  \Gamma(x)\beta_{\varepsilon}(u^{\varepsilon})& \mbox{Uniformly elliptic operator}\\
 \Div(|\nabla u^{\varepsilon}|^{p-2} \nabla u^{\varepsilon})  &=&  \beta_{\varepsilon}(u^{\varepsilon}) + f_{\varepsilon}(x)& \mbox{p-Laplacian} \\
\Div\left(\frac{g(|\nabla u^{\varepsilon}|)}{|\nabla u^{\varepsilon}|} \nabla u^{\varepsilon}\right) & = & \beta_{\varepsilon}(u^{\varepsilon}) &\mbox{g-Laplacian in Orlicz-Sobolev spaces}\\
\Div(|\nabla u^{\varepsilon}|^{p_{\varepsilon}(x)-2} \nabla u^{\varepsilon})  &=&  \beta_{\varepsilon}(u^{\varepsilon}) + f_{\varepsilon}(x) & \mbox{p(x)-Laplacian}\\
u_{\varepsilon} & \ge & 0 & \mbox{in} \,\,\,\Omega\\
\Delta u^{\varepsilon} & = & \beta_{\varepsilon}(u^{\varepsilon}) + f_{\varepsilon}(x) & \text{Two-phase problem for Laplacian}\\
\Delta u^{\varepsilon} - u_t^{\varepsilon}& = & \beta_{\varepsilon}(u^{\varepsilon}) + f_{\varepsilon}(x) & \text{Two-phase problem for Heat operator}\\
\Div(|\nabla u^{\varepsilon}|^{p-2} \nabla u^{\varepsilon})  &=&  \beta_{\varepsilon}(u^{\varepsilon}) & \mbox{Two-phase problem for p-Laplacian} \\
\end{array}
\right.
$$}}

In the last two decades, non-linear FBPs like \eqref{EqFBP} were widely studied in the literature via singular perturbation methods. In contrast with their variational counterpart (cf. \cite{AC81}, \cite{DPS03}, \cite{LW16}, \cite{MW09}, \cite{MorTei07} and \cite{T1}), the analysis of non-variational singularly perturbed PDEs imposes significant challenging obstacles, mainly due to lack of monotonicity formulae (cf. \cite{LW08} and \cite{LW10}), energy estimates (cf. \cite{DPS03} and \cite{Kar19}) and a stable notion of ``weak formulation'' of solutions (see, \cite{T1}), just to cite a few.

In this scenario, Teixeira in \cite{T0} started the journey of investigation into fully non-linear elliptic singular PDEs as follows
\begin{equation}\label{EqSingPertFN}
    F(x, D^2 u^{\varepsilon}) = \beta_{\varepsilon}(u^{\varepsilon}) \quad \mbox{in} \quad \Omega \quad \text{with} \quad u^{\varepsilon} \ge 0,
\end{equation}
where $\beta_{\varepsilon}(u^{\varepsilon}) \to \delta_0$ (the Dirac delta measure). The author proves optimal Lipschitz regularity of solutions of \eqref{EqSingPertFN}, as well as $H^1$ compactness for Bellman's singular PDEs. Thereafter, in \cite{RT11} the authors finish the analysis introduced in \cite{T0}. In effect, they prove, among other analytic and geometric properties, that the free boundary condition is driven by a new operator, namely $F^{\ast}$, \textit{the recession profile}, which arises via a blow-up argument on the family of elliptic equations generated by the original operator $F$ (we recommend to reader to \cite{RTU19} for the parabolic counterpart of such studies). On the sequence, \cite[Theorem 1.3]{RS2} obtain global Lipschitz regularity estimates to
\begin{equation}\label{UnifEllipticPBP}
\left\{
\begin{array}{rclcl}
F(x, \nabla u^{\varepsilon}, D^2 u^{\varepsilon}) & = & \beta_{\varepsilon}(u^{\varepsilon}) & \mbox{in} & \Omega\\
u^{\varepsilon}(x) & \ge & 0  & \mbox{in} & \Omega\\
u^{\varepsilon}(x) & = & g(x) & \mbox{on} & \partial \Omega,\\
\end{array}
\right.
\end{equation}
and \cite{MoWan1} studied a FBP like \eqref{UnifEllipticPBP} with an inhomogeneous forcing term like \eqref{zeta}. Finally, in \cite{ART17}, the authors prove similar existence, optimal regularity and geometric results for the class of fully nonlinear, anisotropic degenerate elliptic FBPs (with (A0)-(A1), \eqref{Cond1 zeta} and \eqref{nondeg_RT} in force) as follows
$$
  |\nabla u^{\varepsilon}|^{p}F(D^2 u^{\varepsilon}) = \zeta_{\varepsilon}(x, u^{\varepsilon}) \quad \mbox{in} \quad \Omega \quad \mbox{with} \quad p\ge 0 \,\,\,\text{and}\,\,\,u^{\varepsilon} \ge 0 .
$$
Such current results summarize studies on singularly perturbed non-variational PDEs.

\subsubsection{Recent progresses on degenerate equations in non-divergence form}
Now, we turn our attention towards regularity features of our model operator in \eqref{Equation Pe}. Regarding regularity properties to non-variational scenario, more specifically, fully non-linear models with single degeneracy structure
\begin{equation}\label{Eqp-Deg}
   \mathcal{G}_p(x, \nabla u, D^2 u) \defeq   |\nabla u|^pF(x, D^2 u) \quad \text{with} \quad 0< p< \infty,
\end{equation}
have been an increasing focus of studies over the last decades due to their intrinsic connection to several qualitative/quantitative issues in pure mathematics (see, \cite{ART15}, \cite{BD04}, \cite{BD07}, \cite{BD1}, \cite{BD2}, \cite{BD3}, \cite{BDL19} and \cite{IS}), as well as a number of geometric and FBPs (see, \cite{ART17}, \cite{daSLR20}, \cite{DSVI}, and \cite{DSVII}). Additionally, we also refer the interested reader to \cite{C1}, \cite{CC95}, \cite{UserG}, \cite{daSR19}, \cite{DeSFS15}, \cite{Ev82}, \cite{Kry82}, \cite{Kry83}, \cite{MorTei07}, \cite{RS2}, \cite{RT11}, \cite{ST15}, \cite{T0}, \cite{Tru83}, \cite{Tru84} for an incomplete list of corresponding results in the uniformly elliptic scenario.

In turn, in contrast with \eqref{Eqp-Deg}, one of the main characteristics of the model case
\begin{equation}\label{EqModel}
   u\mapsto  \left[|\nabla u|^p+\mathfrak{a}(x)| \nabla u|^q\right] \Delta u \quad (\text{with \eqref{1.3} in force})
\end{equation}
is its transition between two distinct degeneracy rates, which depends on the values of the modulating function $\mathfrak{a}(\cdot)$. For this reason, the diffusion process presents a non-uniformly elliptic and doubly degenerate signature, which mixes up two different $p-$Laplacian type operators in non-divergence form (cf. \cite{ART15}, \cite{BD1}, \cite{BD2}, \cite{BD3}, \cite{BDL19}  and \cite{IS}). Such a prototype in \eqref{EqModel} can be understood as a non-variational extension of certain variational integrals of the calculus of variations with $(p, q)-$growth conditions as follows
\begin{equation}\label{DPF}\tag{\bf DPF}
 \displaystyle   \left(W_0^{1,p}(\Omega)+g, L^m(\Omega)\right) \ni (w, f) \mapsto \text{min}\int_{\Omega} \left(\frac{1}{p}|\nabla w|^p+\frac{\mathfrak{a}(x)}{q}|\nabla w|^q-fw\right)dx,
\end{equation}
where $\mathfrak{a}\in C^{0, \alpha}(\Omega,[0, \infty))$, for some $0< \alpha \leq 1<p\le q< \infty$ and $m \in (N, \infty]$. Finally, notice that minimizers to \eqref{DPF} exhibits non-uniform and doubly degenerate ellipticity in a model with a kind of $(p, q)-$structure:
$$
\Div(\mathcal{A}(x, \nabla u)\nabla u) = f(x) \quad \text{in} \quad \Omega, \quad \text{where} \quad \mathcal{A}(x, \xi) \defeq |\xi|^{p-2}+\mathfrak{a}(x)|\xi|^{q-2}
$$
We recommend seeing \cite{BCM15III}, \cite{DeFM19I}, \cite{DeFM19II}, \cite{DeFM20}, \cite{DeFO} and \cite{LD18} for interesting related works.

Now, let us come back to non-variational models like \eqref{EqModel}. Regarding regularity estimates of fully non-linear models with non-homogeneous degeneracy, the starting point was De Filippis' manuscript \cite{DeF20}, where $C_{\text{loc}}^{1, \alpha}-$regularity for viscosity solutions of
$$
  \left[|\nabla u|^p+\mathfrak{a}(x)|\nabla u|^q\right]F(D^2 u) = f \in L^{\infty}(\Omega), \quad (\text{with (A0)-(A1) and \eqref{1.3} in force})
$$
was addressed, for some $\alpha \in (0, 1)$ depending on universal parameters. In the sequence, \cite{daSR20} establishes sharp gradient estimates to general models driven by \eqref{MainOperator}, as well as a number of applications of such estimates in geometric free boundary and related non-linear elliptic PDEs (cf. \cite{daSLR20} and \cite{DSVI}).

At this point, a natural question arises: what are the regularity and geometric features of solutions and level surfaces to problems of the type \eqref{Equation Pe}? Particularly, we are looking for geometric properties that are independents of the regularization parameter and therefore allow to be carried over (in a uniform fashion) in the limit process.

To the best of the authors' acknowledgment, few advances are known concerning the regularity theory for inhomogeneous FBPs like \eqref{EqFBP}. As a matter of fact, many of these are available in the context of linear operators (cf. \cite{AC81}, \cite{LW08} and \cite{LW10}) and in the uniformly elliptic scenario (cf. \cite{RT11}). We must quote a recent work \cite{DeSFS15}, where the authors deal with an inhomogeneous two-phase FBP driven by fully nonlinear elliptic operators. Finally, further results on the limiting FBPs \eqref{EqFBP}, which include particularly the regularity of the free boundary, are challenging and open issues in such a line of investigation.

A schedule for developing the theory of \eqref{EqFBP} is summarized as follows:

\begin{enumerate}
  \item[\checkmark] Existence, uniform/geometric regularity estimates for certain regularizing solutions of \eqref{Equation Pe}.
  \item[\checkmark] Existence and optimal regularity estimates of certain solutions of \eqref{EqFBP}, e.g., viscosity solutions obtained as a limit of singular perturbation problems.
  \item[\checkmark] Measure theoretic properties of the free boundary, such as finite perimeter and density features for the positivity region.
  \item[\checkmark] Strong regularity properties of the interfaces, e.g. Lipschitz or ``flat'' interfaces became ``regular'' enough (cf. \cite[Ch. 4 and 5]{CS05}) - In forthcoming work, see \cite{daSRRV21}.
\end{enumerate}

\section{Background results}\label{DefPreRes}

In the sequel, we will state an essential tool we will make use of, namely the fundamental estimate from \cite[Theorem 1.1]{daSR20} and \cite{DeF20}. For this reason, we recall the following $C_{\text{loc}}^{1, \alpha}$ regularity result.

\begin{theorem}[{\bf $C_{\text{loc}}^{1, \alpha}-$estimates}]\label{GradThm} Let $F$ be an operator satisfying (A0)-(A2). Suppose further assumptions \eqref{1.2} and \eqref{1.3} are in force. Let $u$ be a bounded viscosity solution to
$$
 \mathcal{G}(x, \nabla u, D^2u) = f(x, u) \in L^{\infty}(\Omega \times \R).
$$
Then, $u \in C_{\text{loc}}^{1, \alpha}(\Omega)$. Moreover, the following estimate holds true
\begin{equation*}
\|u\|_{C^{1,\alpha}\left(\Omega^{\prime}\right)}\leq C\cdot \left(\|u\|_{L^{\infty}(\Omega)} +1+ \|f\|_{L^{\infty}(\Omega\times \R)}^{\frac{1}{p+1}}\right)
\end{equation*}
for universal constants $\alpha \in (0, 1)$ and $C>0$.
\end{theorem}

\begin{remark}\label{Rem1} In the sequence, we recall some pivotal estimates, we will make use of through this manuscript. Precisely, if $u$ is a non-negative viscosity solution to
\begin{equation}\label{EqPrinc}
  \mathcal{G}(x, \nabla u, D^2 u) = f \in C^{0}(\Omega)
\end{equation}
and the assumptions (A0)-(A2), \eqref{1.2} and \eqref{1.3} there hold. Then, we have

\begin{itemize}
  \item[(1)] {\bf Harnack inequality}: If $f \in L^m(B_1)\cap C^0(B_1)$ with $m>N$, then
  $$
  \displaystyle \sup_{B_{1/2}} u(x) \leq C(N, \lambda, \Lambda, p, q, L_1)\cdot \left\{\inf_{B_{1/2}} u(x) + \max\left\{\left\|\frac{f}{1+ \mathfrak{a}}\right\|_{L^{\infty}(B_1)}^{\frac{1}{p+1}}, \left\|\frac{f}{1+ \mathfrak{a}}\right\|_{L^{\infty}(B_1)}^{\frac{1}{q+1}}\right\}\right\}.
  $$

    \item[(2)] {\bf Gradient estimates}: If $f \in L^{\infty}(B_1)$, then, $u \in C_{\text{loc}}^{1, \alpha}(B_1)$ and
    $$
    |\nabla u(0)| \leq C(N, \lambda, \Lambda, p, \alpha, L_1, L_2, \|F\|_{C^{\omega}}, \|\mathfrak{a}\|_{L^{\infty}})\cdot \left( \|u\|_{L^{\infty}(B_1)} + 1+ \|f\|_{L^{\infty}(B_1)}^{\frac{1}{p+1}}\right)
    $$
\end{itemize}
\end{remark}

In the sequel, we have a kind of ``Cutting Lemma'', which strongly relies on \cite[Lemma 6]{IS} and is concerned with the homogeneous, doubly degenerate problem. We refer the reader to \cite[Lemma 4.1]{DeF20}, which, on the one hand, is not the precise statement of such a result, but on the other hand, it can be inferred from \cite[Lemma 4.1]{DeF20} (see also \cite[Lemma 6]{IS}) by a careful inspection of the proof.

\begin{lemma}[{\bf Cutting Lemma}]\label{cutting}
Let $F$ be an operator satisfying (A0)-(A2), \eqref{1.2} and \eqref{1.3}, and $u$ be a viscosity solution of
\[
\mathcal{H}(x, Du)F(x, D^2u)=0 \quad \textrm{ in }\quad B_1(0).
\]
Then $u$ is viscosity solution of
\[
   F(x,D^2u)=0 \quad \textrm{ in }\quad B_1(0).
\]
\end{lemma}

Now, let us present a useful comparison tool. For that purpose, we shall assume the following: There exists a continuous function $\widehat{\omega}:[0, \infty) \to [0, \infty)$ with $\widehat{\omega}(0) = 0$, such that if $\mathrm{X}, \mathrm{Y} \in \text{Syn}(N)$ and $\varsigma \in (0, \infty)$ fulfill
\begin{equation}\label{e2}
-\varsigma
\left(
\begin{array}{cc}
\mathrm{Id}_{\mathrm{N}} & 0 \\
0 & \mathrm{Id}_{\mathrm{N}} \\
\end{array}
\right)
\leq
\left(
\begin{array}{cc}
\mathrm{X} & 0 \\
0 & \mathrm{Y} \\
\end{array}
\right)
\leq
4 \varsigma
\left(
\begin{array}{cc}
\mathrm{Id}_{\mathrm{N}} & -\mathrm{Id}_{\mathrm{N}} \\
-\mathrm{Id}_{\mathrm{N}} & \mathrm{Id}_{\mathrm{N}} \\
\end{array}
\right),
\end{equation}
then
\begin{equation}\label{EqComPrinc}
\mathcal{G}(x, \varsigma(x-y), \mathrm{X})-\mathcal{G}(y, \varsigma(x-y), -\mathrm{Y})\le \widehat{\omega}(\varsigma |x-y|^2) \quad \forall\,\,x, y \in \R^N, x \neq y.
\end{equation}

We stress that such a condition is not necessary when $\mathcal{G}$ does not depend on $x-$variable. In this context, conditions (A0) and (A1) are sufficient to our purpose (cf. \cite{UserG}).

The proof of Comparison Principle holds the same ideas as ones in \cite[Theorem 1.1]{BD04} and \cite[Theorem 1]{BD07}. For this reason, we will omit the proof here.

\begin{lemma}[{\bf Comparison Principle}]\label{comparison principle}
	Assume that assumptions (A0)-(A1), \eqref{1.2}, \eqref{1.3}, \eqref{e2} and \eqref{EqComPrinc} there hold. Let $f \in C^0(\bar{\Omega})$ and $h$ be a continuous increasing function satisfying $h(0) = 0$. Suppose $u_1$ and $u_2$ are respectively a viscosity supersolution and 	subsolution of
	$$
	\mathcal{G}(x, \nabla w, D^2w) = h(w) + f(x) \quad\text{in} \quad \Omega.
	$$
If $u_1 \geq u_2$ on $\partial \Omega$, then $u_1 \geq u_2$ in $\Omega$.

Furthermore, if $h$ is nondecreasing (in particular if $h \equiv 0$), the result holds if $u_1$ is a strict supersolution or vice versa if $u_2$ is a strict subsolution.
\end{lemma}

Finally, we present a qualitative property known as ABP estimate (cf. \cite{DFQ} and \cite{JunMio10}).

\begin{theorem}[{\bf Alexandroff-Bakelman-Pucci estimate}]
  Assume that assumptions (A0)-(A2) there hold. Then, there exists $C = C(N, \lambda, p, q, \diam(\Omega))>0$ such that for any $u \in C^0(\overline{\Omega})$ viscosity sub-solution (resp. super-solution) of \eqref{EqPrinc} in $\{x \in \Omega:u(x)>0\}$ (resp. $\{x \in \Omega:u(x)<0\}$, satisfies
$$
\displaystyle \sup_{\Omega} u(x) \leq \sup_{\partial \Omega} u^{+}(x) +C\cdot\diam(\Omega)\max\left\{\left\|\frac{f^{-}}{1+\mathfrak{a}}\right\|^{\frac{1}{p+1}}_{L^N(\Gamma^{+}(u^{+}))}, \left\|\frac{f^{-}}{1+\mathfrak{a}}\right\|^{\frac{1}{q+1}}_{L^N(\Gamma^{+}(u^{+}))}\right\},
$$
{\small{
$$
\left(\text{resp.} \,\,\,\displaystyle \sup_{\Omega} u^{-}(x) \leq \sup_{\partial \Omega} u^{-}(x) +C\cdot\diam(\Omega)\max\left\{\left\|\frac{f^{+}}{1+\mathfrak{a}}\right\|^{\frac{1}{p+1}}_{L^N(\Gamma^{+}(u^{-}))}, \left\|\frac{f^{+}}{1+\mathfrak{a}}\right\|^{\frac{1}{q+1}}_{L^N(\Gamma^{+}(u^{-}))}
\right\}\right)
$$}}
where $\Gamma^{+}(u) \defeq \left\{x \in \Omega: \exists \,\, \xi \in \R^N\,\,\,\text{such that}\,\,u(y)\le u(x)+\langle \xi, y-x\rangle \,\, \forall\, y \in \Omega\right\}$.

\end{theorem}


Let us finish this section by commenting on how to construct viscosity solutions. The idea is to obtain a solution of the Perron type, the least super-solution. Our approach holds by adapting the so-called method of sub-solutions and super-solution to the viscosity theory to produce a solution.

Given a regular boundary datum $g$, a pair of sub-solution and super-solution solutions can be obtained by solving the
\begin{equation} \label{G1}
	\mathcal{G}(x, \nabla \underline{u}^{\varepsilon}, D^2 \underline{u}^{\varepsilon}) =\sup_{\Omega \times [0,+\infty)}  \zeta_{\varepsilon}(x, t) \quad \textrm{and} \quad \mathcal{G}(x, \nabla \overline{u}^{\varepsilon}, D^2 \overline{u}^{\varepsilon}) = \inf_{\Omega \times [0,+\infty)}  \zeta_{\varepsilon}(x, t)
\end{equation}
satisfying $\underline{u}^{\varepsilon}=\overline{u}^{\varepsilon}=g$ on $\partial \Omega$. We stress that the existence of solutions of \eqref{G1} for instance follows from ideas in \cite[Proposition 2 and 3]{BD07-2}.

Finally, fixed a pair of sub-solution and super-solution solutions of the equation \eqref{Equation Pe}, the following general procedure yields the existence of Perron’s solution:

\begin{theorem}\label{Perron}
Let  $\mathcal{G}$ be an elliptic fully non-linear operator satisfying (A0)-(A2), \eqref{1.2} and \eqref{1.3}, and $h \in C^{0, 1}(\Omega \times [0, \infty))$ be a bounded, Lipschitz function in $\mathbb{R}^N$. Suppose that the equation
$$
	 \mathcal{G}(x,\nabla u, D^2 u) = h(x,u)
$$
admits  $u_{\star}, u^{\star} \in C^{0}(\overline{\Omega})$ sub-solution and super-solution, respectively, such that $u_{\star} \le u^{\star}$ in $\Omega$   and  $u_{\star}=u^{\star} = g \in C^{0}(\partial \Omega)$. Define the set of functions,
$$
	\mathscr{S} \defeq \{w \in C(\overline{\Omega}) ; u_{\star} \le w \le u^{\star} \,\,\,\, \textrm{and} \,\,\,\, w \,\,\,\, \textrm{super-solution of} \,\,\,\, \mathcal{G}(x, \nabla u, D^2 u) = h(x,u) \}.
$$
Then,
$$
	v(x) \defeq \inf_{w \in \mathscr{S}} w(x)
$$
is a continuous viscosity solution of
$$
\left\{
\begin{array}{rclcl}
 \mathcal{G}(x,\nabla u, D^2 u) &=& h(x,u) & \mbox{in} & \Omega\\
u(x)&=&g(x) & \mbox{on} & \partial \Omega
\end{array}
\right.
$$
\end{theorem}

By using \cite{UserG} and Theorem \ref{HoldEstThm} the proof follows the same lines as the one given by \cite[Theorem 2.1]{ART17}. For this reason, we will omit it.

\section{Lipschitz regularity estimates} \label{Sct Lip}

In this Section, we derive uniform gradient estimates, which in particular provides compactness in the local uniform convergence topology. In view of the results proven in Section \ref{Sct Nondeg}, such an estimate is indeed optimal.

Before starting the proof of the local Lipschitz estimate, we need to ensure the uniform bound for non-negative solutions to \eqref{Equation Pe}. Such a statement is a direct consequence of the Alexandroff-Bakelman-Pucci estimate (see, Theorem \ref{ABPthm}).

\vspace{0.5cm}

\begin{lemma}\label{UnifomBound}
Let $u^{\varepsilon}$ be a non-negative viscosity solution to \eqref{Equation Pe}. Then, there exists a constant $\mathrm{C}(\verb"universal")>0$ such that
$$
    \|u^{\varepsilon}\|_{L^{\infty}(\Omega)} \le \|g\|_{L^{\infty}(\partial \Omega)} + \mathrm{C}\cdot\diam(\Omega)\max\left\{\left\|\frac{\mathcal{B}_0}{1+\mathfrak{a}}\right\|^{\frac{1}{p+1}}_{L^N(\Omega)}, \left\|\frac{\mathcal{B}_0}{1+\mathfrak{a}}\right\|^{\frac{1}{q+1}}_{L^N(\Omega)}\right\}.
$$
\end{lemma}
\begin{proof}
Define $v^{\varepsilon}(x) \defeq u^{\varepsilon}(x) - \|g\|_{L^{\infty}(\partial \Omega)}$. Now, notice that
$$
 \mathcal{G}(x, \nabla  v^{\varepsilon}, D^2 v^{\varepsilon}) \geq \mathcal{B}_0 \quad \mbox{in} \quad \Omega
$$
in the viscosity sense. Moreover $v^{\epsilon}  \leq  0$ on $\partial \Omega$. Therefore, the
 Aleksandrov-Bakelman-Pucci estimate (Theorem \ref{ABPthm}) provides the desired estimate.
\end{proof}

\medskip

\begin{proof}[{\bf Proof of Theorem \ref{thmreg}}] Firstly, we will analyze the transition region $\{0\leq u^\epsilon \leq \epsilon\}\cap \Omega^{\prime}$. Thus, for $\epsilon \le  \min\left\{1, \frac{1}{2}\dist(\Omega^{\prime}, \partial \Omega)\right\}$, fix $x_0 \in \{0\leq u^\epsilon \leq \epsilon\}\cap \Omega^{\prime}$ and define the scaled function
$$
   v(x)\defeq \frac{1}{\epsilon}u^{\epsilon}(x_0 + \epsilon x) \quad \mbox{in} \quad B_1.
$$
It is straightforward to show that $v$ fulfils in the viscosity sense
$$
	\mathcal{G}_{x_{0}, \e}(x, \nabla v(x), D^2 v(x)) = \epsilon \zeta(x_0 + \epsilon x, u^\epsilon(x_0+ \epsilon x)) \quad \textrm{in} \quad B_1,
$$
where (see, equation \eqref{MainOperator})
$$
\left\{
\begin{array}{rcl}
  F_{x_0, \epsilon}(x, \mathrm{X}) & \defeq & \epsilon F\left(x_{0}+\epsilon x,\frac{1}{\epsilon}\mathrm{X}\right) \\
  \mathcal{H}_{x_{0}, \e}(x, \xi) & \defeq &  \mathcal{H}(x_{0}+\epsilon x, \xi)\\
  \mathfrak{a}_{x_0, \e}(x) & \defeq & \mathfrak{a}(x_{0}+\epsilon x)\\
  f_{x_0, \epsilon}(x) & \defeq & \e\zeta(x_0 + \epsilon x, u^\epsilon(x_0+ \epsilon x))
\end{array}
\right.
$$
Hence, it follows from the structural assumption \eqref{Cond1 zeta} that
$$
	0 \le f_{x_0, \epsilon}(x)  \le \mathcal{A} + \mathcal{B} \defeq \mathrm{C}_\star.
$$
Moreover, it is easy to check that the assumptions (A0)-(A2) and \eqref{1.2} and \eqref{1.3} are satisfied to $F_{x_0, \epsilon}$, $\mathcal{H}_{x_{0}, \e}$ and $\mathfrak{a}_{x_0, \e}$ (with the same universal constants). Therefore, from the $C_{\text{loc}}^{1,\alpha}$ regularity estimate (see, Theorem \ref{GradThm} and Remark \ref{Rem1} (Item (2))), we have
\begin{equation}\label{estim1}
	|\nabla v(0)| \le \mathrm{C}\cdot\left\{\|v\|_{L^{\infty}\left(B_{\frac{1}{2}}\right)} +  1 + {\mathrm{C}_\star}^{\frac{1}{p+1}}\right\} \quad (\text{for a constant} \quad \mathrm{C}(\verb"universal")>0),
\end{equation}
Additionally, since $v(0) = \frac{1}{\epsilon} u^{\epsilon}(x_0) \le 1$, by Harnack inequality (Theorem \ref{ThmHarIneq} ) we obtain
\begin{equation}\label{Harnack}
	\|v\|_{L^{\infty}\left(B_{\frac{1}{2}}\right)} \le \mathrm{C}_0(N, \lambda,\Lambda, L_1, \mathfrak{a}, p, q, \mathrm{C}_\star) \quad (\text{for some} \,\,\, \mathrm{C}_0>0\,\,\,\text{independent of} \,\,\,\epsilon).
\end{equation}
Finally, by combining \eqref{estim1} and \eqref{Harnack} we get
\begin{equation}\label{Lip Eq00}
	|\nabla u^{\epsilon}(x_0)| = |\nabla v(0)| \le \mathrm{C}_1, \quad (\text{for some} \,\,\, \mathrm{C}_1>0\,\,\,\text{independent of} \,\,\,\epsilon).
\end{equation}

Now, we proceed in analysing the region $\{u^{\epsilon}> \e\} \cap \Omega^{\prime}$. For that end, let us designate
$$
	\Gamma_{\epsilon} \defeq \{x \in \Omega^{\prime} \suchthat  u^{\epsilon}(x)=\epsilon\},
$$
and fix a point $\hat{x}_0  \in \{u^{\epsilon}> \e\} \cap \Omega^{\prime}$. Next, we evaluate the distance from $\hat{x}_0$ up to the $\Gamma_\epsilon$ and we will label it as $r_0 \defeq \dist(\hat{x}_0, \Gamma_{\epsilon})$. Then, we define the re-normalized function $v_{\hat{x}_0, r_0} \colon B_1 \to \mathbb{R}$ as
$$
	v_{\hat{x}_0, r_0}(x) \defeq \frac{u^{\epsilon}(\hat{x}_0 + r_0x) - \epsilon}{r_0}.
$$
It is easy to check that $v_{\hat{x}_0, r_0}$ satisfies in the viscosity sense
\begin{equation}\label{EqLips1}
  \mathcal{G}_{\hat{x}_0, r_0}(x, \nabla v_{\hat{x}_0, r_0}(x), D^{2}v_{\hat{x}_0, r_0}(x)) = r_0\zeta_{\epsilon}(\hat{x}_0+r_0x,u^{\varepsilon}(\hat{x}_0+r_0x)),
\end{equation}
where as before
$$
\left\{
\begin{array}{rcl}
  F_{\hat{x}_0, r_0}(x, \mathrm{X}) & \defeq & r_0F\left(\hat{x}_0+r_0 x,\frac{1}{r_0}\mathrm{X}\right) \\
  \mathcal{H}_{\hat{x}_0, r_0}(x, \xi) & \defeq &  \mathcal{H}(\hat{x}_0+ r_0 x, \xi)\\
  \mathfrak{a}_{\hat{x}_0, r_0}(x) & \defeq & \mathfrak{a}(\hat{x}_0+r_0x)\\
  f_{\hat{x}_0, r_0}(x) & \defeq & r_0\zeta(\hat{x}_0 + r_0x, u^\epsilon(\hat{x}_0+r_0 x))
\end{array}
\right.
$$

By construction, $u^\epsilon(\hat{x}_0 +r_0x) > \epsilon \,\, \forall\, x \in B_1$. Particularly,
\begin{equation}\label{EqLips2}
v_{\hat{x}_0, r_0}(x) \ge 0 \quad \text{for every} \quad x \in B_1
\end{equation}
Hence, it follows from the assumption \eqref{Cond1 zeta} that
$$
	\|f_{\hat{x}_0, r_0}\|_{L^\infty(B_1)} \le \mathrm{C}_2(\mathcal{B}, \diam(\Omega^{\prime})).
$$
By making use of $C_{\text{loc}}^{1,\alpha}$ regularity estimate (see, Theorem \ref{GradThm} and Remark \ref{Rem1} (Item (2))), we conclude
\begin{equation}\label{Lip Eq01}
	|\nabla u^\epsilon(\hat{x}_0)| = |\nabla v_{\hat{x_0}, r_0}(0)| \le \mathrm{C}\cdot\left(\frac{1}{r_0} \|u^\epsilon - \epsilon \|_{L^\infty\left(B_{\frac{r_0}{2}}(\hat{x}_0)\right)} + 1+ \mathrm{C}_2^{\frac{1}{p+1}} \right).
\end{equation}
It remains to show a uniform control for the term $\frac{1}{r_0} \|u^\epsilon - \epsilon \|_{L^\infty\left(B_{\frac{r_0}{2}}(\hat{x}_0)\right)}$. For that purpose, let $z_0 \in \Gamma_\epsilon$ be a point that achieves distance, i.e., $r_0 = |\hat{x}_0 - z_0|$. Now, from the Lipschitz estimate proven for points within $\{0 \le  u^\epsilon \le \epsilon\} \cap \Omega^{\prime}$, namely sentence\eqref{Lip Eq00}, we have
$$
	|\nabla u^\epsilon(z_0)| \le \mathrm{C}_0.
$$
Hence,
\begin{equation}\label{EqLips3}
\frac{\partial v_{\hat{x_0}, r_0}}{\partial \nu}(y_0) \leq |\nabla u^\epsilon(z_0)|\le  \mathrm{C}_0 \quad \text{with} \quad v_{\hat{x_0}, r_0}(y_0) = 0 \quad \text{and} \quad y_0 \defeq \frac{z_0-\hat{x}_0}{r_0}.
\end{equation}.

Thus, from \eqref{EqLips1}, \eqref{EqLips2} and \eqref{EqLips3} we are able to apply Lemma \ref{lemma2.1}
and conclude that there exists a constant $\mathrm{c}(\verb"universal")>0$ such that
\begin{equation}\label{EqLips4}
   v_{\hat{x_0}, r_0}(0) \le \mathrm{c}.
\end{equation}
Moreover, from Harnack inequality (see, Theorem \ref{ThmHarIneq}) we obtain (using \eqref{EqLips4})
$$
  \displaystyle \frac{1}{r_0} \|u^\epsilon - \epsilon \|_{L^\infty\left(B_{\frac{r_0}{2}}(\hat{x}_0)\right)} = \sup_{B_{\frac{1}{2}}(0)} v_{\hat{x_0}, r_0}(x) \le \mathrm{C}_0(N, \lambda, \Lambda, p, q, L_1, \mathfrak{a}, \mathrm{c}, \mathcal{B}, \diam(\Omega^{\prime}))
$$
which finishes the proof of the Theorem.
\end{proof}

\begin{remark}
  It is important stress that for each $\epsilon>0$ fixed, viscosity solutions $u^\e$ are in effect $C_{\text{loc}}^{1, \alpha}(\Omega)$. Particularly, when $F$ is concave/convex, it follows from \cite[Corollary 1.1]{daSR20} that $u^\e \in C_{\mbox{loc}}^{1, \frac{1}{p+1}}(\Omega)$. At this point, on one hand, for any $0<\varsigma\ll 1$ small, near $\e$-layers, one obtains that
$$
	\lim\limits_{\e \to 0^{+}} \|\nabla u^\e\|_{C^{0, \varsigma}(\Omega^{\prime})} =+\infty.
$$
On the other hand, Theorem \ref{thmreg} one ensures that the Lipschitz norm of $u^\e$ remains uniformly controlled (independently of $\e$). In such a point of view, our estimates are optimal.
\end{remark}

\section{Geometric non-degeneracy} \label{Sct Nondeg}

\subsection{Building Barriers}

As previously explained in the introduction, one of the main intricacies in dealing with singularly perturbed models with non-homogeneous degeneracy is to avoid that solutions degenerate along their transition surfaces. For this reason, a decisive devise for overcoming such an obstacle will be implementing a geometric non-degeneracy estimate.

In this Section, we show that solutions grow in a linear fashion away from $\epsilon$-level surfaces, inside $\{u^\epsilon > \epsilon\}$. In particular, this implies that \textit{in measure} the two free boundaries do not intersect. The proof shall be based on building an appropriate barrier function. To this end, we shall look at elliptic models with non-homogeneous degeneracy as follows
\begin{equation}\label{visc1.1}
	\mathcal{H}(x, \nabla w)\mathscr{M}^{+}_{\lambda,\Lambda}(D^2 w) = \zeta(x, w) \quad \text{in} \quad \mathbb{R}^N \quad (\text{with} \,\,\,\eqref{1.2} \,\,\,\text{and}\,\,\,\eqref{1.3}\,\,\,\text{in force}),
\end{equation}
where the reaction term fulfils the non-degeneracy assumption (cf. \eqref{nondeg_RT}):
\begin{equation}\label{ND e=1}
	\mathscr{I}^{\ast} \defeq \inf\limits_{\mathbb{R}^N \times [t_0, \mathrm{T}_0]} \zeta(x,t) >0,
\end{equation}

\begin{proposition}[{\bf Barrier}] Let $0<t_0<\mathrm{T}_0<1$ be fixed. For a constant $\mathrm{A}_0(\verb"universal")>0$ to be chosen \textit{a posteriori}, there exists a radially symmetric profile $\Theta_{\mathrm{L}} \colon \mathbb{R}^N \rightarrow \mathbb{R}$ fulfilling:

\begin{itemize}
  \item[(1)] $\Theta_{\mathrm{L}} \in C^{1,1}_{\textrm{loc}}(\mathbb{R}^N)$;
  \item[(2)] $t_0 \le \Theta_{\mathrm{L}}(x) \le \mathrm{T}_0$;
  \item[(3)] $\Theta_{\mathrm{L}}$ is a (point-wise) super-solution to \eqref{visc1.1};
  \item[(4)] For some $\kappa_0(\verb"universal")>0$
\begin{equation}\label{nondeg}
\Theta_{\mathrm{L}}(x) \geq \kappa_0\cdot 4\mathrm{L} \quad \mbox{for} \quad |x| \geq 4\mathrm{L}, \quad \text{where}\,\,\, \mathrm{L} \geq \mathrm{L}_0 \defeq \sqrt{\frac{\mathrm{T}_0-t_0}{\mathrm{A}_0}}.
\end{equation}
\end{itemize}

\end{proposition}

\begin{proof}
For an $\alpha(\verb"universal")>0$ to be chosen (\textit{a posteriori}) let us define
\begin{equation}\label{barrier}
\Theta_{\mathrm{L}}(x)\defeq
\left\{
\begin{array}{lcl}
t_0 & \mbox{for} & 0\leq |x| < \mathrm{L}; \\
\mathrm{A}_0\left(|x| - \mathrm{L}\right)^2+t_0 & \mbox{for} & \mathrm{L}\leq |x| < \mathrm{L}+\sqrt{\frac{\mathrm{T}_0-t_0}{\mathrm{A}_0}}; \\
\psi(\mathrm{L})-\frac{\phi(\mathrm{L})}{|x|^{\alpha}} & \mbox{for} & |x| \geq \mathrm{L}+\sqrt{\frac{\mathrm{T}_0-t_0}{\mathrm{A}_0}}
\end{array}
\right.
\\
\end{equation}
where
\begin{equation}\label{choices}
\left\{
\begin{array}{lcl}
  \phi(\mathrm{L}) & \defeq & \displaystyle \frac{2}{\alpha}\sqrt{(\mathrm{T}_0-t_0)\mathrm{A}_0}\left(\mathrm{L}+\sqrt{\frac{\mathrm{T}_0-t_0}{\mathrm{A}_0}} \right)^{1+\alpha} \\
  \psi(\mathrm{L}) & \defeq & \displaystyle \mathrm{T}_0+\phi(\mathrm{L})\left(\mathrm{L}+\sqrt{\frac{\mathrm{T}_0-t_0}{\mathrm{A}_0}} \right)^{-\alpha},
\end{array}
\right.
\end{equation}

It is easy to check that $\Theta_{\mathrm{L}} \in C^{1,1}_{\textrm{loc}}(\mathbb{R}^N)$. For this reason, we may compute the second order derivatives of $\Theta_{\mathrm{L}}$ a.e. Moreover, from definition $t_0 \le \Theta_{\mathrm{L}}(x) \le \mathrm{T}_0$ is easily verified.

In the sequel, we are going to show that $\Theta_{\mathrm{L}}$ satisfies (point-wise) \eqref{visc1.1},
as long as we perform the appropriate choices under the parameters $\alpha, \mathrm{A}_0>0$.

In effect, for $0\leq |x| < \mathrm{L}$ such an inequality is clearly satisfied (due to \eqref{ND e=1}).

In the annular region $\mathrm{L}\leq |x| < \mathrm{L}+\sqrt{\frac{\mathrm{T}_0-t_0}{\mathrm{A}_0}}$, we obtain
$$
\left\{
\begin{array}{rcl}
 \displaystyle |\nabla\Theta_{\mathrm{L}}(x)|& = & \displaystyle 2\mathrm{A}_0\left(|x|-\mathrm{L} \right)\\
  & \leq & 2\sqrt{\mathrm{A}_0(\mathrm{T}_0-t_0)} \\
   D^2 \Theta_{\mathrm{L}}(x) & = & \displaystyle 2\mathrm{A}_0\left[ \left(\frac{1}{|x|^2} - \frac{(|x|-\mathrm{L})}{|x|^3}\right) x \otimes x + \frac{(|x|-\mathrm{L})}{|x|} \textrm{Id}_{\mathrm{N}} \right]\\
   & \leq & 4\mathrm{A}_0 \cdot \textrm{Id}_{\mathrm{N}}.
\end{array}
\right.
$$

Therefore, by using \eqref{1.2} and \eqref{1.3} we obtain
{\small{
$$
\mathcal{H}(x, \nabla \Theta_{\mathrm{L}}(x))\mathscr{M}^{+}_{\lambda,\Lambda}(D^2 \Theta_{\mathrm{L}}(x))  \leq 4\mathrm{A}_0N\Lambda L_2\left[\left(2\sqrt{\mathrm{A}_0(\mathrm{T}_0-t_0)}\right)^p + \|\mathfrak{a}\|_{L^{\infty}(\Omega)}\left(2\sqrt{\mathrm{A}_0(\mathrm{T}_0-t_0)}\right)^q\right].
$$}}

Now, thanks to the assumption \eqref{ND e=1} we are able to choose a positive constant $\mathrm{A}_0 = \mathrm{A}_0(N, \Lambda, L_2, p, q, \|\mathfrak{a}\|_{L^{\infty}(\Omega)}, \mathrm{T}_0-t_0, \mathscr{I}^{\ast})$ such that
$$
\mathcal{H}(x, \nabla \Theta_{\mathrm{L}}(x))\mathscr{M}^{+}_{\lambda,\Lambda}(D^2 \Theta_{\mathrm{L}}(x))  \leq \mathscr{I}^{\ast}.
$$

Thus, by using the Item (2) we conclude that
$$
\mathcal{H}(x, \nabla \Theta_{\mathrm{L}}(x))\mathscr{M}^{+}(D^2 \Theta_{\mathrm{L}}(x)) \leq \mathscr{I}^{\ast}\leq \zeta(x, \Theta_{\mathrm{L}}(x)).
$$

Finally, let us analyse the region $|x| \geq \mathrm{L}+\sqrt{\frac{\mathrm{T}_0-t_0}{\mathrm{A}_0}}$. Straightforward calculation gives
$$
D^2 \Theta_{\mathrm{L}}(x) =\alpha\phi(\mathrm{L})|x|^{-(\alpha+2)}\left( -\frac{(\alpha+2)}{|x|^2}x\otimes x+ \text{Id}_{\mathrm{N}}\right).
$$
Thus,
$$
\mathscr{M}^{+}_{\lambda,\Lambda}(D^2 \Theta_{\mathrm{L}}(x)) \leq \alpha\phi(\mathrm{L})|x|^{-(\alpha+2)}\left[ -(\alpha+1)\lambda + (N-1)\Lambda \right].
$$
Therefore, selecting $\alpha \in \left[(N-1)\frac{\Lambda}{\lambda} -1, \,\infty\right)$, we get (using \eqref{ND e=1} and $\mathcal{H}(x, \xi)\ge 0$)
$$
	\mathcal{H}(x, \nabla \Theta_{\mathrm{L}}(x))\mathscr{M}^{+}(D^2 \Theta_{\mathrm{L}}(x)) \le 0 \le \zeta(x,\Theta_{\mathrm{L}}(x)),
$$
which assure us that $\Theta_{\mathrm{L}}$ fulfils \eqref{visc1.1} as desired.

In conclusion, we will show that the super-solution $\Theta_L$ fulfils \eqref{nondeg}

At this point, from \eqref{choices} (second sentence) we have
$$
|x| \geq 4\mathrm{L} \geq 2(\mathrm{L}+\mathrm{L}_0) = 2\, \left(\dfrac{\phi(\mathrm{L})}{\psi(\mathrm{L})-\mathrm{T}_0}\right)^{\frac{1}{\alpha}}.
$$
Hence, for $\alpha >0$
$$
   \Theta_{\mathrm{L}}(x) = \psi(\mathrm{L})-\frac{\phi(\mathrm{L})}{|x|^{\alpha}} \geq \psi(\mathrm{L}) - \frac{1}{2^{\alpha}}(\psi(\mathrm{L})-\mathrm{T}_0) > \frac{1}{2^{\alpha}}(\psi(\mathrm{L})-\mathrm{T}_0).
$$
Therefore, by using \eqref{choices}
$$
\Theta_{\mathrm{L}}(x) \geq \kappa_0\cdot 4\mathrm{L} \quad \text{for} \quad  \kappa_0 \defeq \frac{\alpha^{-1}}{2^{\alpha+1}}\sqrt{\mathrm{A}_0(\mathrm{T}_0-t_0)}.
$$

\end{proof}

\subsection{Linear Growth}\label{SecLinGrowth}

In order to establish lower bounds on the growth of solutions to \eqref{Equation Pe} inward the set $\{u^\epsilon > \epsilon \}$, the strategy will be to consider appropriate scaling versions of the universal barrier $\Theta_{\mathrm{L}}$.

\begin{proof}[{\bf Proof of Theorem \ref{cresc}}] Let us assume, without loss of generality, $0 \in \{u^{\epsilon} > \epsilon\}$. Now, we set $\eta \defeq \dfrac{d_\epsilon(0)}{2}$ and consider the reaction term
$$
	\zeta(z,t) \defeq \left \{
		\begin{array}{ll}
			\epsilon \zeta_\epsilon(\epsilon x, \epsilon t), &\text{if }  \epsilon x \in \Omega \\
			\mathscr{I}^{\ast} & \text{otherwise}.
		\end{array}
	\right.
$$
Given the barrier $\Theta_{\mathrm{L}}$ built-up previously, we define
$$
	\Theta_{\epsilon}(x) \defeq \epsilon \cdot \Theta_{\frac{\eta}{4\epsilon}} \left(\frac{x}{\epsilon} \right).
$$

Next, one verifies that the scaled barrier $\Theta_{\epsilon}$ fulfils
$$
	\mathcal{G}(x, \nabla \Theta_{\epsilon}(x), D^2 \Theta_\epsilon(x)) \le \zeta_\epsilon(x, \Theta_\epsilon(x)),
$$
Moreover, by \eqref{nondeg} and \eqref{barrier} we verify that for $4\mathrm{L}_0 \epsilon \ll \eta$,
\begin{equation}\label{ebarrier}
	\Theta_\epsilon(0) = t_0\cdot \epsilon \quad \mbox{and} \quad 	\Theta_\epsilon(x)\geq \kappa_0\cdot \eta \quad \text{on} \quad \partial B_{\eta}.
\end{equation}

In the sequel, we claim that there exists a $z_0\in \partial B_{\eta}$ such that
\begin{equation}\label{z}
\Theta_{\epsilon}(z_0) \le u^{\epsilon}(z_0).
\end{equation}
Indeed, if we assume $\Theta_{\epsilon} > u^{\epsilon}$ everywhere in $\partial B_{\eta}$, then
$$
	v^{\epsilon}(x)\defeq \min\{\Theta_{\epsilon}(x),\,u^{\epsilon}(x)\}
$$
would be a super-solution to \eqref{Equation Pe}. However, $v^\epsilon$ is strictly below of $u^{\epsilon}$, which contradicts the minimality of $u^\epsilon$.
Therefore, by \eqref{ebarrier} and \eqref{z}, we conclude
\begin{equation}\label{strong}
	 \sup\limits_{\overline{B_\eta}}u^\epsilon(x) \ge u^{\epsilon}(z_0) > \Theta_{\epsilon}(z_0) \ge \kappa_0\cdot \eta.
\end{equation}
Furthermore, $u^\epsilon$ solves in the viscosity sense
$$
\mathcal{B}_0 \leq \mathcal{G}(x, \nabla u^\epsilon, D^2 u^\epsilon) \leq \mathcal{B} \quad \mbox{in} \quad B_{2\eta}.
$$
Therefore, by Harnack inequality (see, Theorem \ref{ThmHarIneq} and Remark \ref{HarnIneqSV}), we obtain
$$
\sup\limits_{B_{\eta}} u^\epsilon \leq \mathrm{C}(N, \lambda, \Lambda, q, L_1) \cdot \left(u^\epsilon(0)+ \max\left\{(2\eta)^{\frac{p+2}{p+1}}\mathcal{B}^\frac{1}{p+1}, (2\eta)^{\frac{q+2}{q+1}}\mathcal{B}^\frac{1}{q+1}\right\} \right),
$$
Thus, by \eqref{strong},
$$
u^\epsilon(0) \geq \left(\mathrm{C}^{-1}\kappa_0 - 2^{\frac{q+2}{q+1}}\max\left\{ \mathcal{B}^\frac{1}{p+1}\eta^{\frac{1}{p+1}}, \mathcal{B}^\frac{1}{q+1}\eta^{\frac{1}{q+1}}\right\}\right)\eta.
$$
Finally, by taking
$$
  0< \eta <\min\left\{\mathcal{B}\cdot\left(\mathrm{C}^{-1}\kappa_02^{-\frac{q+2}{q+1}}\right)^{p+1}, \mathcal{B}\cdot\left(\mathrm{C}^{-1}\kappa_02^{-\frac{q+2}{q+1}}\right)^{q+1}, \frac{\diam(\Omega)}{4}\right\},
$$
we have
$$
   u^\epsilon(0) \geq \mathrm{c}\cdot\eta.
$$
for some constant $0<\mathrm{c}(\verb"universal")< \mathrm{C}^{-1}\kappa_0$.
\end{proof}

\section{Some important implications from Theorems \ref{thmreg} and \ref{cresc}}

In this section, we discuss some implications of the sharp control of solutions, established in the sections \ref{Sct Lip} and \ref{SecLinGrowth}.  As a consequence of Lipschitz regularity, i.e. Theorem \ref{thmreg}, and linear growth, i.e. Theorem \ref{cresc}, we obtain the complete control of $u^{\epsilon}$ in terms of $d_{\epsilon}(x_0)$.


\begin{corollary} \label{cor1}
Given $\Omega^{\prime} \Subset \Omega$, there exists a constant $\mathrm{C}(\Omega^{\prime}, \verb"universal parameters")>0$ such that for $x_0 \in \{u^{\epsilon} >\epsilon\} \cap \Omega^{\prime}$ and $0<\epsilon \le  \frac{1}{2}d_{\epsilon}(x_0)$, there holds
$$
	\mathrm{C}^{-1} d_{\epsilon}(x_0) \le u^{\epsilon}(x_0) \le \mathrm{C}\, d_{\epsilon}(x_0).
$$
\end{corollary}
\begin{proof}
Take $z_0 \in \partial \{u^{\epsilon} > \epsilon\}$, such that $|z_0-x_0| = d_{\epsilon}(x_0)$. Thus, follow from Theorem \ref{thmreg},
$$
	u^{\epsilon}(x_0) \le \mathrm{C}_0\,d_{\epsilon}(x_0) + u^\epsilon(z_0) \le (\mathrm{C}_0+1) \,\, d_{\epsilon}(x_0),
$$
The first inequality is precisely the statement of Theorem \ref{cresc}.
\end{proof}

Next we will prove that Perron type solutions are strongly non-degenerate near $\epsilon$-layers. It means that the $\displaystyle \sup_{B_r(x_0)} u^{\epsilon}$ (for $x_0 \in \{u^{\epsilon} > \epsilon\} \cap \Omega^{\prime}$) is comparable to $r$. This is an important piece of information about the growth rate of $u^{\epsilon}$ away from $\epsilon$-surfaces. The proof is, to a certain extent, a scholium of the proof of Theorem \ref{cresc}.

\begin{proof}[{\bf Proof of Theorem \ref{degforte}}]
Firstly, the estimate from above follows directly from Lipschitz regularity (Theorem \ref{thmreg}). Now, as in the Theorem \ref{cresc}, we take $\Theta_\e(x) = \e \Theta_{\frac{\rho}{4\e}}(x)$. Thus,
$$
	u^{\epsilon}(z_0) > \Theta_{\epsilon}(z_0),
$$
for some point $z_0 \in \partial B_{\rho}(x_0)$.  Finally, we note that
$$
\displaystyle\sup_{\overline{B_\rho(x_0)}} u^{\epsilon}(x) \ge	u^{\epsilon}(z) > \Theta_{\epsilon}(z_0) \ge \kappa_0 \cdot \rho.
$$
\end{proof}

\begin{corollary}\label{densidade}
Given $x_0\in \{u^\epsilon>\epsilon\} \cap \Omega^{\prime}$, $\epsilon \ll \rho$ and $\rho\ll 1$ small enough (in a universal way), there exists a constant $0< \mathrm{c}_0(\verb"universal") < 1$ such that
$$
	\Leb(B_{\rho}(x_0) \cap \{u^{\epsilon} > \epsilon\}) \ge \mathrm{c}_0\cdot \Leb(B_{\rho}(x_0))\,
$$
where $\Leb(\mathrm{S})$ is the Lebesgue measure of the set $\mathrm{S} \subset \R^N$.
\end{corollary}
\begin{proof}

From Strong non-degeneracy (Theorem \ref{degforte}) that there exists $y_0 \in  B_{\rho}(x_0)$ such that
$$
	u^{\epsilon}(y_0) \ge \mathrm{c}_0\rho.
$$
From Lipschitz regularity (Theorem \ref{thmreg}), for $z_0 \in B_{\kappa \rho}(y_0)$, we have
$$
	u^{\epsilon}(z_0) -\mathrm{C} \kappa \rho \ge u^{\epsilon}(y_0).
$$
Thus, by previous estimates, it is possible to choose $0 < \kappa \ll 1$ small (in a universal way) such that
$$
	z \in B_{\kappa \rho}(y_0) \cap B_{\rho}(x_0) \quad \textrm{and} \quad u^{\epsilon}(z) > \epsilon.
$$
Finally, there exists a portion of $B_\rho(x_0)$ with volume in order $\sim\rho^N$ within $\{u^\epsilon>\epsilon\}$. Thus, we verify
$$
	\Leb(B_{\rho}(x_0) \cap \{u^{\epsilon} > \e\}) \ge \Leb(B_{\rho}(x_0) \cap B_{\kappa \rho}(y_0)) = \mathrm{c}_0 \,\Leb(B_{\rho}(x_0)),
$$
for some constant  $0<\mathrm{c}_0(\verb"universal")\ll 1$.
\end{proof}

\begin{corollary}\label{intdens}
Given $x_0\in \{u^\epsilon>\epsilon\} \cap \Omega^{\prime}$, $\epsilon \ll \rho$ and $\rho\ll 1$ small enough (in a universal way), then
$$
   \frac{1}{\rho}\intav{\; B_{\rho}(x_0)} u^\epsilon(x) dx \geq \mathrm{c}
$$
for a constant $\mathrm{c}(\verb"universal")>0$ also depending on $\epsilon$.
\end{corollary}
\begin{proof}
As in Corollary \ref{densidade}, there exists a constant $0<\kappa(\verb"universal")\ll 1$, such that
$$
   \intav{\; B_{\rho}(x_0)} u^\epsilon(x) dx \geq C_N\intav{\; B_{\rho}(x_0)\cap B_{\kappa \rho}(y_0)} u^\epsilon(x)dx \geq \mathrm{c}\,\rho
$$
for a constant $0<\mathrm{c}(\verb"universal")\ll 1$ and some $y_0 \in \{u^\epsilon>\epsilon\} \cap \Omega^{\prime}$.
\end{proof}

\subsection{A Harnack type inequality}

\hspace{0.3cm} For solutions of \eqref{Equation Pe} a classical Harnack inequality is valid for balls that touch the free boundary along the $\varepsilon$-surfaces, i.e., $\partial\{u^{\varepsilon} > \varepsilon\} \cap \Omega^{\prime}$.

\begin{theorem}\label{ThmHarIneq} Let $u^\varepsilon$ be a solution of \eqref{Equation Pe}. Let also $x_0 \in \{u^{\varepsilon} > \varepsilon\}$ and $\varepsilon\leq d \defeq d_{\varepsilon}(x_0)$. Then,
$$
\displaystyle \sup_{B_{\frac{d}{2}}(x_0)} u^{\varepsilon}(x) \leq \mathrm{C} \cdot \inf_{B_{\frac{d}{2}}(x_0)} u^{\varepsilon}(x)
$$
for a constant $\mathrm{C}(\verb"universal")>0$ independent of $\varepsilon$.
\end{theorem}

\begin{proof} Let $z_1, z_2 \in \overline{B_{\frac{d}{2}}(x_0)}$ be points such that
$$
\displaystyle \inf_{B_{\frac{d}{2}}(x_0)} u^{\varepsilon}(x) = u^{\varepsilon}(z_1) \quad \mbox{and} \quad \sup_{B_{\frac{d}{2}}(x_0)} u^{\varepsilon}(x) = u^{\varepsilon}(z_2).
$$
Since $d_{\varepsilon}(z_1) \geq \frac{d}{2}$, by Corollary \ref{cor1}
\begin{equation}\label{eqHar6.1}
u^{\varepsilon}(z_1) \geq \mathrm{C}_1 \cdot d.
\end{equation}
Moreover, by Strong non-degeneracy \ref{degforte}
\begin{equation}\label{eqHar6.2}
u^{\varepsilon}(z_2) \leq \mathrm{C}_2 \cdot \left(\frac{d}{2} + u^{\varepsilon}(x_0)\right).
\end{equation}

Next, by taking $y_0 \in \partial \{u^{\varepsilon} > \varepsilon\}$ such that $d=|x_0-y_0|$ and $z\in\overline{B_d(y_0)}\cap\{u^\varepsilon>\varepsilon\}$, we obtain from Corollary \ref{cor1} and Theorem \ref{degforte}
\begin{equation}\label{eqHar6.3}
u^{\varepsilon}(x_0) \leq \sup\limits_{B_d(z)}u^\varepsilon \leq \mathrm{C}_2 \cdot (d+ u^{\varepsilon}(z)) \leq \mathrm{C}_3 \cdot d.
\end{equation}

In conclusion, by combining \eqref{eqHar6.1}, \eqref{eqHar6.2} and \eqref{eqHar6.3}, we obtain
$$
   \displaystyle \sup_{B_{\frac{d}{2}}(x_0)} u^{\varepsilon}(x) \leq \mathrm{C} \cdot \inf_{B_{\frac{d}{2}}(x_0)} u^{\varepsilon}(x).
$$
\end{proof}

\subsection{Porosity of the level surfaces}\label{porous}

\hspace{0.3cm} As a consequence of the growth rate and the non-degeneracy (Theorems \ref{thmreg} and \ref{cresc}), we obtain porosity of level sets.

\begin{definition}[{\bf Porous set}]\label{d5.1}
A set $\mathrm{S}\subset\mathbb{R}^n$ is denominated porous with porosity $\delta>0$, if $\exists\,\mathrm{R}>0$ such that
$$
   \forall x\in \mathrm{S}, \,\,\,\forall r\in(0,\mathrm{R}),\,\,\,\exists y\in\mathbb{R}^N\,\textrm{ such that }\,B_{\delta r}(y)\subset B_r(x)\setminus \mathrm{S}.
$$
\end{definition}

A porous set of porosity $\delta$ has Hausdorff dimension not exceeding $N-\mathrm{c}\delta^N$, where $\mathrm{c}=\mathrm{c}(N)>0$ is a dimensional constant. Particularly, a porous set has Lebesgue measure zero (see,\cite{Z88}).

\begin{theorem}[{\bf Porosity}]\label{t5.2}
Let $u^\varepsilon$ be a solution of \eqref{Equation Pe}. Then, the level sets $\partial \{u^{\varepsilon} >\varepsilon\} \cap \Omega^{\prime}$ are porous with porosity constant independent of $\varepsilon$.
\end{theorem}

\begin{proof}
Let $\mathrm{R}>0$ and $x_0\in\Omega^{\prime} \Subset \Omega$ be such that $\overline{B_{4\mathrm{R}}(x_0)}\subset\Omega$.

\begin{itemize}
  \item[{\bf CLAIM:}] The set $\partial\{u^{\varepsilon}>\varepsilon\} \cap B_{\mathrm{R}}(x_0)$ is porous
\end{itemize}

Let $x\in \partial\{u^{\varepsilon}>\varepsilon\} \cap B_{\mathrm{R}}(x_0)$ be fixed. For each $r\in(0,\mathrm{R})$ we have $\overline{B_r(x)}\subset B_{2\mathrm{R}}(x_0)\subset\Omega$. Now, let $y\in\partial B_r(x)$ such that $u^{\varepsilon}(y)=\sup\limits_{\partial B_r(x)}u^{\varepsilon}(x)$. From Strong non-degeneracy (Theorem \ref{degforte})
\begin{equation}\label{5.1}
    u^{\varepsilon}(y)\geq \mathrm{c}\cdot r.
\end{equation}
On the other hand, we know that (see, Theorem \ref{thmreg}) near the free boundary
\begin{equation}\label{5.2}
    u^{\varepsilon}(y)\leq \mathrm{C}\cdot d_{\varepsilon}(y),
\end{equation}
where $d_{\varepsilon}(y)$ is the distance of $y$ from the set $\overline{B_{2\mathrm{R}}(x_0)}\cap\Gamma_{\varepsilon}$. Next, from \eqref{5.1} and \eqref{5.2} we get
\begin{equation}\label{5.3}
    d_{\varepsilon}(y)\geq\delta r
\end{equation}
for a positive constant $\delta \in (0, 1)$.

Let now $y^{\ast}\in[x,y]$ (straight line segment connecting the points $x$ and $y$) be such that $|y-y^{\ast}|=\frac{\delta r}{2}$. Hence, we see that
\begin{equation}\label{5.4}
   B_{\frac{\delta}{2}r}(y^*)\subset B_{\delta r}(y)\cap B_r(x).
\end{equation}
In effect, for each $z\in B_{\frac{\delta}{2}r}(y^{\ast})$
$$
   |z-y|\leq |z-y^{\ast}|+|y-y^{\ast}|<\frac{\delta r}{2}+\frac{\delta r}{2}=\delta r,
$$
and
$$
   |z-x|\leq|z-y^{\ast}|+\big(|x-y|-|y^{\ast}-y|\big)<\frac{\delta r}{2}+\left(r-\frac{\delta r}{2}\right)=r,
$$
then \eqref{5.4} follows.

Now, from \eqref{5.3} $B_{\delta r}(y)\subset B_{d_{\varepsilon}(y)}(y)\subset\{u^{\varepsilon}>\varepsilon\}$, then
$$
   B_{\delta r}(y)\cap B_r(x)\subset\{u^{\varepsilon}>\varepsilon\},
$$
which provides together with \eqref{5.4} the following
$$
   B_{\frac{\delta}{2}r}(y^*)\subset B_{\delta r}(y)\cap B_r(x)\subset B_r(x)\setminus\partial\{u_{\varepsilon}>\varepsilon\}\subset B_r(x)\setminus ( \partial\{u^{\varepsilon}>\varepsilon\} \cap B_{\mathrm{R}}(x_0)),
$$
thereby finishing the proof.
\end{proof}

It is important to stress that for FBPs modeled by a merely uniformly elliptic operator, one should not expect an improved Hausdorff estimate for the interface. In turn, when diffusion is driven by Laplacian operator, then Alt-Caffarelli's theory in \cite{AC81} establishes that $c\delta^N = 1$. At this point, a natural issue is what is the minimum structural assumption under $F$ as to obtain perimeter estimates of
the interface. We will address an answer to such a question in the next Section.

\section{Hausdorff measure estimates}\label{Sct HE}

In this section, we establish Hausdorff measure estimates of the approximating level surfaces. A necessary condition for the study of such an estimate is to impose the non-degeneracy of the reaction term propagates up to the transition layer. Hereafter, in condition \eqref{nondeg_RT}, we shall take $t_0=0$, i.e.,
$$
	\mathscr{I} \defeq  \inf\limits_{\Omega \times [0,\mathrm{T}_0]}  \epsilon \zeta_\epsilon(x, \epsilon t)>0
$$
for some $\mathrm{T}_0>0$ will be enforced. A condition at infinity on the governing operator $F$ is also required in our analysis concerning Hausdorff estimates, which will be discussed soon.

Next result states that, in measure, the Hessian of an approximating solution blows-up near the transition layer as $\epsilon \to 0^{+}$. The proof follows the same lines as \cite[Proposition 6.1]{ART17} and \cite[Proposition 5.1]{RTS17}. For this reason, we will omit it here.

\begin{proposition}\label{ac1}
Fix $\Omega^{\prime} \Subset \Omega$, $\mathrm{C} \gg 1$ and $\rho <  \dist(\Omega^{\prime},\partial\Omega)$. There exists $\epsilon_0>0$ such that, for $\epsilon\leq \epsilon_0$ there holds
\begin{equation}
\int_{B_\rho(x_\epsilon)}\left(\zeta_\epsilon(x, u^\epsilon)-\mathrm{C}\right)dx\geq 0 \quad \text{for any} \quad  x_\epsilon\in \partial\{u^\epsilon>\epsilon\}\cap \Omega^{\prime}.
\end{equation}
\end{proposition}

Mathematically, Proposition \ref{ac1} implies that near the transition layer, the governing operator $F$ gets evaluated at very large matrices. Such an insight motivates the following structural asymptotic condition on the nonlinearity:

\begin{definition}[{\bf Asymptotic Concavity}] A uniformly elliptic operator $F: \Omega \times \textit{Sym}(N) \to \R$ is said to fulfil the $\mathrm{C}_F-$\textit{Asymptotic Concavity Property} (resp. Asymptotic Convexity Property) if there exists $\mathfrak{A} \in \mathcal{A}_{\lambda,\Lambda}$ and non-negative continuous function $\mathrm{C}^{\ast}_F: \Omega \to \R$ such that
\begin{equation}\label{ACP} \tag{{\bf ACP}}
    F(x, X) \leq \tr(\mathfrak{A}(x) \cdot X) + \mathrm{C}^{\ast}_F(x) \quad (\text{resp.} \,\,F(x, X) \geq \cdots),
\end{equation}
for all $(x, X)\in \Omega \times Sym(N)$, where
$$
   \mathcal{A}_{\lambda,\Lambda}\defeq \left\{\mathfrak{A} \in \textrm{Sym}(N) \suchthat \lambda \text{Id}_{N} \le \mathfrak{A}\le \Lambda \text{Id}_{N}\right\}.
$$
\end{definition}

It is noteworthy that the \eqref{ACP} condition is weaker than the concavity (resp. convexity) one, which, for instance it is required for instance in Evans-Krylov-Trudinger's $C_{\text{loc}}^{2,\alpha}$ regularity theory (see,  \cite{Ev82}, \cite{Kry82}, \cite{Kry83}, \cite{Tru83} and \cite{Tru84}). Indeed, it means a sort of concavity (resp. convexity) condition at the infinity of $\text{Sym}(N)$ for $F$.
Furthermore, the concavity (resp. convexity) assumption is precisely when $\mathrm{C}_F = 0$. In a geometric viewpoint, such a condition means that for each $x \in \Omega$ fixed, there exists a hyperplane which
decomposes $\R \times Sym(N)$ in two semi-spaces such that the graph of $F(x, \cdot)$ is always below
such a hyperplane (see \cite{ART17}, \cite{RTS17}, \cite{RT11} and references therein for some motivations
and other details).

In turn, if $F$ fulfils \eqref{ACP}, then its recession operator is a concave (resp. convex) operator, in other words
{\small{
$$
\begin{array}{rcl}
  \displaystyle F^{\ast}(x, X) = \displaystyle\lim_{\tau \to 0+} \tau F \left(x, \frac{1}{\tau} X \right) & \leq &  \displaystyle \lim_{\tau \to 0+} \left[\tr(\mathfrak{A}(x) \cdot X) +
   \tau \mathrm{C}^{\ast}_F(x)\right]\\
   & = & \tr(\mathfrak{A}(x) \cdot X) \quad (\text{resp.} \geq \tr(\mathbb{A}(x) \cdot X)).
\end{array}
$$}}

\begin{example}
Let us consider $F:\textit{Sym}(N) \to \R$ a $C^1$ uniformly elliptic operator. Then, its \textit{recession profile} $F^{\ast}$ should be understood as the ``limiting operator'' for the natural scaling on $F$. By way of illustration, for a number of operators, it is possible to verify the existence of the following limit
$$
  \mathfrak{A}_{ij} \defeq \lim_{\|X\|\to \infty} \frac{\partial F}{X_{ij}}(X),
$$
In such a context, we obtain $F^{\ast}(X) = \tr(\mathfrak{A}_{ij}X)$. An interesting example is the class of Hessian type operators:
$$
\displaystyle  F_m(e_1(D^2 u), \cdots, e_N(D^2 u)) \defeq \sum_{j=1}^{N} \sqrt[m]{1+e_j(D^2 u)^m}-N,
$$
where $m \in \mathbb{N}$ is an odd number. In this setting,
$$
  \displaystyle F^{\ast}(X) = \sum_{j=1}^{N} e_j(X) \quad (\text{the Laplacian operator}).
$$

We recommend the reader to \cite[Example 2.4]{daSR19}, \cite[Section 6.2]{daSR20}, \cite[Example 3.6]{DSVI} and \cite[Example 5]{DSVII} for a number of other illustrative examples.
\end{example}

Recently, improved regularity estimates for viscosity solutions of \textit{Asymptotically Concave} equations were proven in \cite{ST15} (see also \cite{daSR19} for global regularity results). We also emphasize that such operators play an essential role in establishing finiteness of $(N-1)-$Hausdorff measure in several fully non-linear singularly perturbed FBPs, whose Hessian of solutions blows-up through the phase transition. For
this reason, the limiting free boundary condition is ruled by the $F^{\ast}$ rather than $F$ (see, \cite{RT11} for an illustrating example).

Finally, hereafter in this Section, we assume the governing operator $F$ has the asymptotically concave property \eqref{ACP}.

\begin{remark}\label{RemACP}
Notice that if $u^{\epsilon}$ is a Perron's solution to \eqref{Equation Pe}, then we may verify in the viscosity sense
$$
	F(x, D^2 u^{\epsilon}) = \zeta_{\epsilon}(x, u^{\epsilon}).\mathcal{H}\left(x, \nabla u^{\epsilon}\right)^{-1}\quad \textrm{in} \quad \{u^{\epsilon} > \epsilon\} \cap \Omega^{\prime},
$$
for any $\Omega^{\prime} \Subset \Omega$. Hence, by Lipschitz regularity (Theorem \ref{thmreg}) and \eqref{N-HDeg}, one has
$$
	F(x, D^2 u^{\epsilon}) = \zeta_{\epsilon}(x, u^{\epsilon})\mathcal{H}\left(x, \nabla u^{\epsilon}\right)^{-1} \ge \zeta_{\epsilon}(x, u^{\epsilon})L_2^{-1}(C_0^p+ \|\mathfrak{a}\|_{L^{\infty}(\Omega)}C_0^q)^{-1}.
$$

Therefore, by the \ref{ACP} condition
\begin{eqnarray*}
\int_{B_\rho(x_\varepsilon)}\mathfrak{A}_{ij}(x)\,D_{ij}u^\e(x)\,dx &\geq& \int_{B_\rho(x_\varepsilon)} \left[\zeta_\varepsilon(x, u^\varepsilon)\mathcal{H}(x, \nabla u^{\epsilon})^{-1}-\mathrm{C}^{\ast}_{F}(x)\right]\,dx\\
&\ge& \mathrm{C}_{p, q}^{-1}\int_{B_\rho(x_\varepsilon)} \left[\zeta_\varepsilon(x, u^\varepsilon)-\|\mathrm{C}^{\ast}_F\|_{L^{\infty}(\Omega)}\mathrm{C}_{p, q} \right]\,dx\\
&>&0,
\end{eqnarray*}
where $\mathrm{C}_{p, q} \defeq L_2(C_0^p+ \|\mathfrak{a}\|_{L^{\infty}(\Omega)}C_0^q)$ and we have used the Proposition \ref{ac1}.

\end{remark}

In the next result, Remark \ref{RemACP} will allow us to adapt some arguments available for elliptic linear problems (cf. \cite{AC81}). At this point, the proof can be obtained following the same ideas as ones in \cite[Lemma 6.3]{ART17} and \cite[Lemma 4.1]{RT11}.

\begin{lemma}\label{intreg}
There exists a constant $\mathrm{C}(\Omega^{\prime}, \verb"universal parameters")>0$ such that, for each $x_\epsilon\in \partial\{u^\epsilon>\epsilon\}\cap \Omega^{\prime}$ and $\rho\ll 1$, there holds
\[
\displaystyle\int\limits_{B_\rho(x_\epsilon)\cap \{\epsilon\leq u^\epsilon < \mu\}} |\nabla u^\epsilon|\,^2\,dx \leq \mathrm{C}\mu\rho^{N-1}.
\]
\end{lemma}

Next, we will remember some definitions and auxiliary results.
\begin{definition}[{\bf $\delta$-density}]
Given an open subset $\mathcal{O} \subset \mathbb{R}^N$, we say that $\mathcal{O}$ has the $\delta$-density property in $\Omega$ for $0 < \delta <1$, if there exists $\tau>0$ such that
$$
   \Leb(B_{\delta}(x) \cap \mathcal{O}) \ge \tau\cdot \Leb(B_{\delta}(x)).
$$
\end{definition}

\begin{definition}[{\bf $\delta$-neighborhood of a set}]
  Given a measurable set $\mathrm{S} \subset \mathbb{R}^N$ and a positive constant $\delta>0$, we denote:
$$
	\mathcal{N}_\delta (\mathrm{S}) \defeq \{x \in \mathbb{R}^N \mid \dist(x, \mathrm{S})<\delta\},
$$
the $\delta$-neighborhood of $\mathrm{S}$ in $\mathbb{R}^d$.
\end{definition}

Next we will introduce the notion of the Hausdorff measure.

\begin{definition}[{\bf $\mathcal{H}^{j}$-Hausdorff measure}] Let $r_0 > 0$ be given, $0 < \delta < r_0$ be fixed, and let $\mathrm{S} \subset \R^N$ be a given Borel set. For an arbitrary $j \in \mathbb{N}\setminus\{0\}$, we define the $(\delta, j)$-Hausdorff content of $\mathrm{S}$ as follows:
$$
\displaystyle  \mathcal{H}^j_{\delta}(\mathrm{S}) \defeq \inf\left\{\sum_{i} r_i^j: \,\,\,\mathrm{S} \subset \bigcup_{i} B_{r_i}(x_i) \,\,\,\text{such that}\,\,\,r_i< \delta\right\},
$$
where the infimum is taken over all covers $\{B_{r_i}(x_i)\}_{i}$ of $\mathrm{S}$. Therefore, the $\mathcal{H}^{j}$-Hausdorff measure of $\mathrm{S}$ is defined as:
$$
  \displaystyle \mathcal{H}^j(\mathrm{S}) \defeq \lim_{\delta \to 0^{+}} \mathcal{H}^j_{\delta}(\mathrm{S}).
$$
\end{definition}

Before establishing uniform bounds of the $\mathcal{H}^{N-1}$-Hausdorff measure of the level-surfaces $\partial\{u^\epsilon>\epsilon\}$, let us recall a classical result from measure theory.

\begin{lemma}[{\bf Density property}]\label{tec_lemma}
Given an open set $\mathcal{O}\Subset \Omega$, there holds:
\begin{enumerate}
\item[(1)] If there exists $\delta$ such that $\mathcal{O}$ has the $\delta-$density property, then there exists a constant $\mathrm{C}=\mathrm{C}(\tau, N)$, where:
\[
|\mathcal{N}_\delta(\partial \mathcal{O})\cap B_\rho(x)|\leq \frac{1}{2^N\tau}|\mathcal{N}_\delta(\partial \mathcal{O})\cap B_\rho(x)\cap \mathcal{O}|+\mathrm{C}\delta\rho^{N-1}
\]
with $x\in\partial \mathcal{O}\cap\Omega$ and $\delta \ll \rho$.
\item[(2)] If $\mathcal{O}$ has uniform density in $\Omega$ along $\mathcal{O}$, then $|\partial A\cap \Omega|=0$.
\end{enumerate}
\end{lemma}
\begin{proof}
 Property (1) holds by using a covering argument and (2) is a consequence of the Lebesgue Differentiation Theorem (see, \cite{EG92}).
\end{proof}

Next, we obtain an $N$-dimensional measure estimate on $\e$-level layers, that are uniform with respect to  parameter $\e$. The proof holds the same lines as one in \cite[Lemma 6.5]{ART17}. We will omit it here to avoid an unnecessary duplication.

\begin{lemma}\label{lemma4.2}
Fixed $\Omega^{\prime} \Subset \Omega$, there exists a constant $\mathrm{C}^{\ast}(\Omega^{\prime}, \verb"universal parameters")>0$, such that if $ \mathrm{C}^{\ast}\mu \le 2 \rho \ll \dist(\Omega^{\prime}, \partial \Omega)$  then, for $ \mu,\epsilon>0$ small enough, with $3\mathrm{C}_1 \epsilon < \mu \ll \rho$,  we have
\[
\Leb\left(\{\mathrm{C}_1\epsilon <u^\epsilon<\mu\}\cap B_\rho(x_\epsilon)\right)\leq \mathrm{C}^{\ast}\mu\rho^{N-1},
\]
where $x_\epsilon\in \partial \{u^{\epsilon} >\epsilon\} \cap \Omega^{\prime}$, with $d_{\epsilon}(x_\epsilon) \ll \dist(\Omega^{\prime}, \partial \Omega)$ and $\mathrm{C}_1 >1$.
\end{lemma}

Finally, we are ready to establish the $(N-1)$-Hausdorff estimate of approximating level sets (uniform with respect to the parameter $\e$).

\begin{theorem}\label{Hausdf}
Fixed $\Omega^{\prime} \Subset \Omega$, there exists a constant $\mathrm{C}^{\ast}(\Omega^{\prime}, \verb"universal parameters")>0$, such that
$$
\Leb\left(\mathcal{N}_{\mu}(\{\mathrm{C}_1\epsilon<u^\epsilon\})\cap B_\rho(x_\epsilon)\right) \leq \mathrm{C} \mu \rho^{N-1},
$$
for $C_1>1$, $x_\epsilon \in \partial\{\mathrm{C}_1\epsilon < u^\epsilon\}\cap \Omega^{\prime}$, $d_{\epsilon}(x_{\epsilon}) \ll \dist(\Omega^{\prime}, \partial \Omega)$ and $\mathrm{C}_1 \epsilon \ll \rho$.  Particularly,
\begin{equation}\label{eq10}
	\mathcal{H}^{N-1}( \{u^\epsilon = \mathrm{C}_1\epsilon\} \cap B_{\rho}(x_0)) \le \mathrm{C}\cdot \rho^{N-1},
\end{equation}
for constants $\mathrm{C}, \mathrm{C}_1>0$ independent of $\e$.
\end{theorem}

\begin{proof}
From Lipschitz regularity (Theorem \ref{thmreg}), for $z \in \partial\{\mathrm{C}_1\epsilon<u^\epsilon\}$ and $y\in \mathcal{N}_\delta(\partial\{\mathrm{C}_1\epsilon<u^\epsilon\})\cap B_\rho(x_\epsilon)\cap \{\mathrm{C}_1\epsilon<u^\epsilon\}$, we obtain
$$
u^\epsilon(y) \leq u^\epsilon(z)+C|z-y|\leq \mu + \mathrm{C}\delta \leq \kappa\mu,
$$
for $\mu = \mathrm{C}\, \delta$ and $\kappa(\verb"universal")>0$. Therefore, the inclusion
\begin{equation} \label{D1}
\mathcal{N}_\delta(\partial\{\mathrm{C}_1\epsilon<u^\epsilon\})\cap B_\rho(x_\epsilon)\cap \{C_1\epsilon<u^\epsilon\}
\subset
\{\mathrm{C}_1\epsilon<u^\epsilon<\kappa\mu\}\cap B_\rho(x_\epsilon)
\end{equation}
there holds. On the other hand, by Corollary \ref{densidade} and by taking $\delta$ as above, we verify that
$$
	\Leb(B_{\delta}(x) \cap \{u^{\epsilon} > \mathrm{C}_1 \epsilon\}) \ge
	\mathrm{c}\cdot \Leb(B_{\delta}(x)) \quad \textrm{for} \quad x \in \partial \{u^{\epsilon} > \epsilon\}.
$$
Hence, we conclude that $\partial \{u^{\epsilon} > \mathrm{C}_1 \epsilon\}$ has the $\delta$-density property. Thus, Lemma \ref{tec_lemma} ensures the existence of a constant $\mathrm{M}(\verb"universal")>0$ such that
\begin{equation}\nonumber
\begin{array}{ccl}
\Leb(\mathcal{N}_\delta(\partial\{\mathrm{C}_1\epsilon < u^\epsilon\})\cap B_\rho(x_{\epsilon}))  \!\!\!\! & \leq & \!\!\!\! \displaystyle
\mathrm{C}_2\,\Leb(\mathcal{N}_\delta(\partial\{\mathrm{C}_1\epsilon<u^\epsilon\})\cap B_\rho(x_\epsilon)\cap \{\mathrm{C}_1\epsilon<u^\epsilon\})  \\
 & + & \!\!\!\! \mathrm{M}\delta\rho^{N-1}.
\end{array}
\end{equation}
Hence, by applying \eqref{D1},  we obtain
$$
\Leb(\mathcal{N}_\delta(\partial\{\mathrm{C}_1\epsilon < u^\epsilon\})\cap B_\rho(x_{\epsilon})) \le  \mathrm{C}_2\, \Leb(\{\mathrm{C}_1\epsilon<u^\epsilon<\kappa\mu\}\cap B_\rho(x_\epsilon)) + \mathrm{M}\delta \rho^{N-1},
$$
for some constant $C_2(\verb"universal")>0$. Finally, for $\mu \ll \rho$, Lemma \ref{lemma4.2} yields
$$
\Leb\left(
	\mathcal{N}_\delta(\partial\{\mathrm{C}_1\epsilon<u^\epsilon\})\cap B_\rho(x_\epsilon)\right) \leq \mathrm{C} \delta \rho^{N-1} \quad \text{for some}  \quad \mathrm{C}>0.
$$

In order to conclude, we take a covering of $\partial \{\mathrm{C}_1 \epsilon < u^{\epsilon}\} \cap B_{\rho}(x_{\epsilon})$ by balls $\{B_{r_j}\}_{j}$ centered at points along $\partial \{\mathrm{C}_1 \epsilon < u^{\epsilon}\} \cap B_{\rho}(x_{\epsilon})$ with radius $\mu \ll 1$. Thus, we may write
$$
	\bigcup_{j} B_{r_j} \subset \mathcal{N}_{\mu}(\{\mathrm{C}_1 \epsilon < u^{\epsilon}\}) \cap B_{\rho + \mu}(x_{\epsilon}).
$$
Therefore, there exist universal constants $\mathrm{C}_3, \mathrm{C}_4 >0$, such that
\begin{eqnarray*}
	\mathcal{H}^{N-1}_{\mu}(\partial \{\mathrm{C}_1 \epsilon < u^{\epsilon}\} \cap B_{\rho}(x_{\epsilon})) &\le& \mathrm{C}_3 \sum_{j} \mathscr{L}^{N-1}(\partial B_{r_j})\\
	&=& \frac{\mathrm{C}_3}{\mu} \sum_{j} \Leb(B_{r_j})\\
	&\le& \frac{\mathrm{C}_4}{\mu} \Leb(\mathcal{N}_{\mu}(\{\mathrm{C}_1 \epsilon < u^{\epsilon}\})\cap B_{\rho + \mu}(x_{\epsilon}))\\
	&\le& \mathrm{C}_4 \mathrm{C} (\rho + \mu)^{N-1}\\
  & = & \mathrm{C}_4\mathrm{C} \rho^{N-1} + \text{o}(1).
\end{eqnarray*}
Finally, we finish the proof of the Theorem by letting $\mu \to 0^{+}$.
\end{proof}

\section{Limiting scenario as $\varepsilon \to 0^{+}$} \label{Sct AL}

 We will establish geometric and measure theoretic properties for a limiting profile $\lim\limits_{j \to \infty}u^{\epsilon_{k_j}}(x)$, for a subsequence $\epsilon_{k_j} \to 0$. In effect, from Lipschitz regularity the family $\{u^{\epsilon_{k}}\}$ is pre-compact in $C^{0,1}_{loc}(\Omega)-$topology. Thus, up to a subsequence, there exists a function $u_0$, obtained as the uniform limit of $u^{\epsilon_{k_j}}$, as $\epsilon_{k_j} \to 0$.

From now on, we will use the following definition when referring to $u_0$:
$$
   u_0(x) \defeq \lim\limits_{j \to \infty}u^{\varepsilon_{k_j}}(x).
$$
Furthermore, we see that such a limiting function verifies
\begin{enumerate}
	\item[(1)] $u_0 \in [0, K_0]$ in $\overline{\Omega}$ for some constant $K_0(\verb"universal")>0$ (independent of $\varepsilon$);
	\item[(2)] $u_0 \in C^{0,1}_{loc}(\Omega)$;
	\item[(3)] $\mathcal{G}(x, \nabla u_0, D^2 u_0) = f_0(x)$ in $\{u_0>0\}$, with $0\le f_0 \in L^{\infty}(\Omega)\cap C^0(\Omega)$.
\end{enumerate}

Notice that by combining item (3) with the regularity estimate established in \cite[Theorem 1.1]{daSR20}, it follows that $u_0 \in C_\text{loc}^{1, \alpha}(\{u_0>0\}).$ However, such an estimate degenerates as we approach $\mathfrak{F}(u_0, \Omega^{\prime})$. Nevertheless, from item (2), the gradient remains under control, even when $\dist(x_0, \mathfrak{F}(u_0, \Omega^{\prime})) \to 0$.

In the sequel, we show that at each point $z_0 \in \mathfrak{F}(u_0, \Omega^{\prime})$, there exists a cone with vertex $z_0$ that confines the graph of the limiting profile.

\begin{proof}[{\bf Proof of Theorem \ref{limite}}]
Firstly, the upper estimate follows from the local Lipschitz continuity of $u_0$. Next, from Corollary \ref{cor1}, there exists $y_{\epsilon} \in \{0 \le u^{\epsilon} \le \epsilon\} \cap \Omega^{\Omega}$ with $d_{\epsilon}(x_0)=|x_0-y_{\epsilon}|$ such that
$$
	u^{\epsilon}(x_0) \ge \mathrm{c}\cdot d_{\epsilon}(x) = \mathrm{c}\, |x_0-y_{\epsilon}|,
$$
for a constant $\mathrm{c}(\verb"universal")>0$. Thus, up to a subsequence, $y_{\epsilon} \to y_0 \in \{u =0\}$ and hence
$$
	u_0(x_0) \ge \mathrm{c}\,|x_0 - y_0| \ge \mathrm{c}\, \dist(x_0,\mathfrak{F}(u_0, \Omega^{\prime})).
$$
\end{proof}

\begin{theorem}\label{limite1}
Given $\Omega^{\prime} \Subset \Omega$, there exist universal positive constants $\mathrm{C}_0$ and $r_0$, such that
$$
	\mathrm{C}_0^{-1} r \le \sup_{B_{r}(x_0)} u_0(x) \le \mathrm{C}_0( r + u_0(x_0))
$$
for any $x_0 \in \overline{\{u_0 >0\}} \cap \Omega^{\prime}$ with $\dist(x_0, \partial \{u_0 >0\}) \le \frac{1}{2}\dist(x_0, \partial \Omega^{\prime})$ and $r \le r_0$.
\end{theorem}

\begin{proof}
 Such estimates are consequences of passing the limit as $\epsilon \to 0$ in Theorem \ref{degforte}.
\end{proof}

Next result states that the set $\{u_0>0\}$ is the limit, in the Hausdorff distance, of $\{u^{\epsilon} > \epsilon\}$ as $\epsilon\to 0$.

\begin{theorem}\label{limite2} Given $\mathrm{C}_1>1$, the following inclusions
\begin{equation}\nonumber
	\{u_0 >0\} \cap \Omega^{\prime} \subset \mathcal{N}_{\delta}(\{u^{\epsilon_k} > \mathrm{C}_1 \epsilon_k\}) \cap \Omega^{\prime}
\quad
\text{and}
\quad
\{u^{\epsilon_k} > \mathrm{C}_1 \epsilon_k\} \cap \Omega^{\prime} \subset \mathcal{N}_{\delta}(\{u_0 > 0\}) \cap \Omega^{\prime},
\end{equation}
hold for $\delta \ll 1$ and $\epsilon_k \ll \delta$.
\end{theorem}

\begin{proof}
 We will prove the first inclusion, since the second one is obtained similarly. Let suppose for sake of contradiction that there exist a subsequence $\epsilon_k\to 0$ and points $x_k \in \{u_0>0\} \cap \Omega^{\prime}$ such that
\begin{equation}\label{dist}
\dist(x_k,\{u^{\epsilon_k}>\mathrm{C}_1\epsilon_k\})>\delta.
\end{equation}
From Theorem \ref{limite1}, and taking $k\gg 1$, we obtain
$$
u^{\epsilon_k}(y_k) = \sup_{B_{\frac{\delta}{2}}(x_k)} u^{\epsilon_k}(x) \geq \frac{1}{2}\cdot\sup_{B_{\frac{\delta}{2}}(x_k)}u_0(x_k) \geq \mathrm{c} \frac{\delta}{2} \geq \mathrm{C}_1\epsilon_k
$$
for some $y_k \in \overline{B_{\frac{\delta}{2}}(x_k)} \cap \{u^{\epsilon_k}>\mathrm{C}_1\epsilon_k\}$, which contradicts \eqref{dist}. This finishes the proof.
 \end{proof}

\begin{theorem} \label{non_deg}
Given a sub-domain $\Omega^{\prime}\Subset \Omega$, there exists a constant $\mathrm{C}(\verb"universal")>0$ and $\rho_0(\Omega^{\prime}, \verb"universal parameters")>0$ such that, for any $x_0 \in \mathfrak{F}(u_0, \Omega^{\prime})$ and $\rho \leq  \rho_0$, there holds
\begin{equation}\label{undeg}
	\mathrm{C}^{-1} \le \frac{1}{\rho}\intav{\;\partial B_{\rho}(x_0)} u_0(x)\; d \mathcal{H}^{N-1} \le \mathrm{C}.
\end{equation}
\end{theorem}

\begin{proof}
From the Lipschitz regularity, the upper estimate is easily satisfied. Now, to prove the another inequality, we consider $z_\epsilon \in \partial \{u_\epsilon > 0 \} \cap \Omega^{\prime}$, satisfying
$$
	|z_\epsilon - x_0| = \dist(x_0, \partial \{u_\epsilon > 0 \}).
$$
Next, from Theorem \ref{limite2}, $z_\epsilon \to x_0$. Thus, we may pass the limit as $\epsilon_k \to 0$ in the thesis of Corollary \ref{intdens}, thereby finishing the Theorem.
\end{proof}

\begin{remark} Taking into account the condition \eqref{undeg}, we will say that $u_0$ is {\it locally  uniformly non-degenerate} in $\mathfrak{F}(u_0, \Omega^{\prime})$. Finally, in some extent, such a property is another way of saying that $u_0$ enjoys Lipschitz regularity and Non-degeneracy property.
\end{remark}

Next, we see that the set $\{u_0>0\}$ has uniform density along $\mathfrak{F}(u_0, \Omega^{\prime})$.

\begin{theorem}\label{densidade4}
 Given $\Omega^{\prime}\Subset \Omega$, there exists a constant $\mathrm{c}_0(\verb"universal")>0$, such that for $x_0 \in \mathfrak{F}(u_0, \Omega^{\prime})$ there holds
\begin{equation}\label{den+}
	\Leb(B_{\rho}(x_0) \cap \{u_0 >0\}) \ge \mathrm{c}_0\Leb(B_{\rho}(x_0)),
\end{equation}
for $\rho \ll 1$. Particularly, $\Leb(\mathfrak{F}(u_0, \Omega^{\prime}))=0$.
\end{theorem}

\begin{proof}
The estimate \eqref{den+} follows as in the proof of Corollary \ref{densidade}. We conclude the result by making use of the Lebesgue differentiation theorem and a covering argument (a Besicovitch–Vitali type result, see \cite{EG92} for details).
\end{proof}


Finally, we are in a position to establish the Hausdorff measure estimate of the limiting free boundary.

\begin{proof}[{\bf Proof of Theorem \ref{limit3}}]

From Theorem \ref{limite2}, for $k \gg 1$ large enough, one has
$$
   \mathcal{N}_{\delta}(\mathfrak{F}(u_0, \Omega^{\prime})) \cap B_{\rho}(x_0)  \subset  \mathcal{N}_{4 \delta}(\partial \{u^{\epsilon_{k}} > \mathrm{C}_1 \epsilon_k\}) \cap B_{2 \rho}(x_0).
$$
Now, by assuming, $\epsilon_k \ll \delta \ll \rho \ll \dist(\Omega^{\prime} , \partial \Omega)$, the hypothesis of Theorem \ref{Hausdf} are verified, thereby implying the following estimate for the $\delta$-neighborhood,
$$
	\Leb(\mathcal{N}_{\delta}(\mathfrak{F}(u_0, \Omega^{\prime})) \cap B_{\rho}(x_0)) \le \mathrm{C} \cdot \delta \rho^{N-1}.
$$

Next, let $\{B_{r_j}\}_{j \in N}$ be a covering of $\mathfrak{F}(u_0, \Omega^{\prime}) \cap B_{\rho}(x_0)$ by balls with radii $\delta>0$ and centered at free boundary points on $\mathfrak{F}(u_0, \Omega^{\prime}) \cap B_{\rho}(x_0)$. Hence,
$$
	\bigcup_{j \in \mathbb{N}} B_{r_j} \subset \mathcal{N}_{\delta}(\mathfrak{F}(u_0, \Omega^{\prime})) \cap B_{\rho + \delta}(x_0).
$$

Therefore, there exists a constant $\overline{\mathrm{C}}(\verb"universal")>0 $ such that
\begin{eqnarray*}
\displaystyle \mathcal{H}^{N-1}_{\delta}(\mathfrak{F}(u_0, \Omega^{\prime}) \cap B_{\rho}(x_0)) &\le& \overline{\mathrm{C}} \sum_{j} \mathscr{L}^{N-1}(\partial B_{r_j})\\
&= & N\frac{\overline{\mathrm{C}}}{\delta} \sum_{j}\Leb(B_{r_j})\\
&\le& N\frac{\overline{\mathrm{C}}}{\delta} \Leb(\mathcal{N}_{\delta}(\mathfrak{F}(u_0, \Omega^{\prime})) \cap B_{\rho+\delta}(x_0))\\
&\le& \mathrm{C}(N) (\rho + \delta)^{N-1}\\
& =& \mathrm{C}(N) \rho^{N-1} + o(\delta).
\end{eqnarray*}
Finally, by letting $\delta \to 0^{+}$ we conclude the proof.

As a consequence of the previous statement we conclude that $\mathfrak{F}(u_0, \Omega^{\prime})$ has locally finite perimeter (see, \cite{EG92} for a precise definition). Furthermore, the reduced free boundary, i.e. $\mathfrak{F}_{\text{red}}(u_0, \Omega^{\prime}) \defeq \partial_{\textrm{red}} \{u_0 >0\} \cap \Omega^{\prime}$ has a total $\mathcal{H}^{N-1}$ measure in the sense that
$$
   \mathcal{H}^{N-1}(\mathfrak{F}(u_0, \Omega^{\prime}) \setminus \mathfrak{F}_{\text{red}}(u_0, \Omega^{\prime})) =0,
$$
We recommend to reader to \cite[Theorem 6.7]{ART17} for the detailed proof. Particularly, the limiting free boundary has an outward vector for $\mathcal{H}^{N-1}$ a. e. in $\mathfrak{F}_{\text{red}}(u_0, \Omega^{\prime})$ (see, \cite{EG92}).
\end{proof}


\subsection{Final comments: An ansatz on the free boundary condition}

In the particular case coming from the homogeneous flame propagation theory:
$$
	\zeta_\epsilon(t) = \dfrac{1}{\epsilon} \zeta\left ( \dfrac{t}{\epsilon}  \right ),
$$
where $\zeta$ is a continuous function  supported in $[0,1]$, then the limiting function satisfies
$$
	F(x, D^2u) = 0 \quad \text{ in } \quad \{u > 0 \},
$$
in view of Cutting Lemma in \cite[Lemma 6]{IS}. In this case, even though the gradient degeneracy is no longer present in the limiting equation, it does leave its {\it signature} on the expected linear behavior along the limiting transition boundary.

Let us analyze one-dimensional profiles, i.e., the limiting configuration of the equation
\begin{equation} \label{profile1}
	\left(|u^{\epsilon}_{x}|^{p} + \kappa|u^{\epsilon}_{x}|^{q}\right)\cdot u^{\epsilon}_{xx}= \zeta_{\e}\left(u^{\epsilon}\right) \quad \text{for} \quad \kappa>0
\end{equation}
By multiplying the above equation by $u^{\epsilon}_{x} dx$, we find the differential equality:
\begin{equation} \label{profile2}
	\left(|u^{\epsilon}_{x}|^p u^{\epsilon}_{x} + \kappa|u^{\epsilon}_{x}|^q u^{\epsilon}_{x}\right)\cdot ( u^{\epsilon}_{xx}dx ) = \zeta_{\e}\left(u^{\epsilon}\right)\cdot u^{\epsilon}_{x}dx.
\end{equation}
However,
$$
	\zeta_{\epsilon}(u^{\epsilon})\cdot u^{\epsilon}_{x}  dx = \frac{d}{dx} \mathfrak{Z}_\epsilon (u^\epsilon),
$$
where
$$
\displaystyle  \mathfrak{Z}_\epsilon(x) \defeq \int_0^{\frac{x}{\epsilon}} \zeta(s) ds \rightarrow \int_{0}^{1} \zeta(s) ds \quad  \text{as} \quad  \epsilon \to 0^{+}.
$$
Performing a change of variables
$$
	u^{\epsilon}_{x}(x) = w(x)  \quad  \Longrightarrow  \quad u^{\epsilon}_{xx} dx = dw,
$$
then we can write
$$
	\int \left(|u^{\epsilon}_{x}|^p u^{\epsilon}_{x} + \kappa|u^{\epsilon}_{x}|^q u^{\epsilon}_{x}\right)\cdot u^{\epsilon}_{xx} dx = \int \left(|w|^{p} + \kappa|w|^{q}\right)wdw.
$$
Thus, computing anti-derivatives in \eqref{profile2} and letting $\epsilon \to 0$, we obtain for the limiting function $u$ that
$$
	  \frac{1}{p+2}|u^{\prime}(x_0)|^{p+2} + \frac{\kappa}{q+2}|u^{\prime}(x_0)|^{q+2}= \int_{0}^{1} \zeta(s)ds.
$$
Therefore,
$$
  |u^{\prime}(x_0)| \leq \min\left\{\sqrt[p+2]{(p+2)\int_{0}^{1} \zeta(s) ds}, \sqrt[q+2]{   \left(\frac{q+2}{\kappa}\right)\int_{0}^{1} \zeta(s) ds}\right\}.
$$

Particularly, by taking $p = 0 = \kappa$, we recover the classical free boundary condition in the isotropic flame propagation theory, see \cite{BCN}.

\section{Appendices}\label{Append}





\subsection{Harnack Inequality}

For the reader's convenience, in this Appendix we gather the statements of two fundamental results in elliptic regularity, namely the Weak Harnack inequality and the Local Maximum Principle. Such pivotal tools will provide a Harnack inequality (resp. local H\"{o}lder regularity) to viscosity solutions.

\begin{theorem}[{\bf Weak Harnack inequality, \cite[Theorem 2]{Imb}}]\label{wharn}
Let $u$ be a non-negative continuous function such that
$$
   F_0(x, \nabla u, D^2u)\leq 0 \quad\textrm{ in }\quad B_1
$$
in the viscosity sense. Assume that $F_0$ is uniformly elliptic in the $X$ variable (see, A1 condition) and $F_0 \in C^0(B_1\times \left(\R^N \setminus B_{\mathrm{M}_{\mathrm{F}}}\right)\times \text{Sym}(N))$ for some $\mathrm{M}_{\mathrm{F}} \geq 0$. Further assume that
\begin{equation}\label{eq.elliwh}
   |\xi|\geq \mathrm{M}_{\mathrm{F}} \quad \text{and} \quad F_0(x, \xi, X)\leq 0 \quad \Longrightarrow \quad \mathscr{M}^-_{\lambda,\Lambda}(X)-\sigma(x)|\xi|-f_0(x)\leq 0.
\end{equation}
for continuous functions $f_0$ and $\sigma$ in $B_1$. Then, for any $q_1 > n$
\[
   \|u\|_{L^{p_0}\left(B_{\frac{1}{4}}\right)}\leq C.\left\{\inf_{B_{\frac{1}{2}}} u+\max\left\{\mathrm{M}_{\mathrm{F}}, \|f_0\|_{L^N(B_1)}\right\}\right\}
\]
for some (universal) $p_0>0$ and a constant $C>0$ depending on $N, q_1, \lambda, \Lambda$ and $\|\sigma\|_{L^{q_1}(B_1)}$.
\end{theorem}

\begin{theorem}[{\bf Local Maximum Principle, \cite[Theorem 3]{Imb}}] \label{localmax}
Let $u$ be a continuous function satisfying
$$
   F_0(x, \nabla u, D^2u)\geq 0 \quad\textrm{ in }\quad B_1
$$
in the viscosity sense. Assume that $F_0$ is uniformly elliptic in the $X$ variable and $F_0 \in C^0(B_1\times \left(\R^N \setminus B_{\mathrm{M}_{\mathrm{F}}}\right)\times \text{Sym}(N))$ for some $\mathrm{M}_{\mathrm{F}} \geq 0$. Further assume that
\begin{equation}\label{eq.elliplmp}
   |\xi|\geq \mathrm{M}_{\mathrm{F}} \quad \text{and} \quad F_0(x, \xi, X)\geq 0 \quad \Longrightarrow \quad \mathscr{M}^+_{\lambda,\Lambda}(X)+\sigma(x)|\xi|+f_0(x)\geq 0.
\end{equation}
for continuous functions $f_0$ and $\sigma$ in $B_1$. Then, for any $p_1>0$ and $q_1 > N$
\[
\sup_{B_{\frac{1}{4}}}u\leq C.\left\{\|u^+\|_{L^{p_1}\left(B_{\frac{1}{2}}\right)}+\max\left\{\mathrm{M}_{\mathrm{F}}, \|f_0\|_{L^N(B_1)}\right\}\right\}
\]
where $C>0$ is a constant depending on $N, q_1, \lambda, \Lambda,\|\sigma\|_{L^{q_1}(B_1)}$ and $p_1$.
\end{theorem}

Let us recall that such results were proved in Imbert's manuscript \cite{Imb} by following the strategy of the uniformly elliptic case, see \cite[Section 4.2]{CC95}. Such a strategy is based on the so-called $L^\varepsilon$-Lemma, which establishes a polynomial decay for the measure of the super-level sets of a non-negative super-solution for the Pucci extremal operator $\mathscr{M}^+_{\lambda,\Lambda}$:
\begin{equation}\label{eq.lep}
\Leb(\left\{x \in B_1:u(x)>t\right\}\cap B_1)\leq \frac{C}{t^{\epsilon}}.
\end{equation}

Unfortunately, Imbert's manuscript has a gap in the proof of \eqref{eq.lep}. Such an error was recently made up in a joint work with Silvestre, see \cite{IS2}, where an appropriate $L^\varepsilon$-estimate was addressed. In fact, their proof holds for ``Pucci extremal operators for large gradients'' defined, for a fixed $\tau$, by:
\[
\widetilde{\mathscr{M}}^+_{\lambda,\Lambda}(D^2u, \nabla u)\defeq
\left\{\begin{array}{ll}
\mathscr{M}^+_{\lambda,\Lambda}(D^2 u)+\Lambda|\nabla u| & \text{ if }|\nabla u|\geq \tau \\
+\infty & \text{ otherwise }
\end{array}\right.
\]
\[
\widetilde{\mathscr{M}}^-_{\lambda,\Lambda}(D^2u, \nabla u)\defeq
\left\{\begin{array}{ll}
\mathscr{M}^-_{\lambda,\Lambda}(D^2 u)-\Lambda|\nabla u| & \text{ if }|\nabla u|\geq \tau \\
-\infty & \text{ otherwise }
\end{array}\right.
\]
The $L^\varepsilon$-estimate is proved to hold whenever $\tau\leq \varepsilon_0$ universal (see, \cite[Theorem 5.1]{IS2}). Moreover, notice that the ellipticity condition $\widetilde{\mathscr{M}}^-_{\lambda,\Lambda}$ is consistent with \eqref{eq.elliwh} if we take $\sigma(x)\equiv\Lambda$. Precisely, if \eqref{eq.elliwh} and $u$ is a super-solution for $F_0$, then it is also a super-solution for $\widetilde{\mathscr{M}}^-_{\lambda,\Lambda}$ with right hand side $f_0$. An analogous reasoning is valid for $\widetilde{\mathscr{M}}^+_{\lambda,\Lambda}$ and \eqref{eq.elliplmp}.

In this point, once the $L^\varepsilon$ is derived, the proof of Theorem \ref{wharn} is exactly as the one in \cite{Imb} which is, in turn, a modification of the uniformly elliptic case in \cite[Theorem 4.8, a]{CC95}. As for Theorem \ref{localmax}, it also follows from \eqref{eq.lep} by assuming (in a fist moment) that the $L^\varepsilon$ norm of $u^+$ is small and the obtaining the general result by interpolation. Indeed, the smallness of the $L^\varepsilon$ norm readily implies \eqref{eq.lep} which in turn gives $u$ is bounded (see, \cite[Lemma 4.4]{CC95}, which is adapted in \cite[Section 7.2]{Imb}).

Notice that our class of operators fits in this scenario by setting
\[
   F_0(x, \nabla v, D^2v) \defeq \mathcal{H}(x, \nabla v)F(x, D^2 v)-f(x).
\]
and
\[
f_0(x)\defeq \frac{L_1^{-1}f^{+}(x)}{\varepsilon_0^p + \mathfrak{a}(x)\varepsilon_0^q} \quad \text{for suitable} \quad \varepsilon_0>0.
\]
In effect, we have that whenever
\[
\mathcal{H}(x, \nabla v)F(x, D^2v)\leq f(x) \quad \text{in} \quad B_1
\]
in the viscosity sense, then the ellipticity condition of $F$ (A1) ensures that
\[
\mathscr{M}^-_{\lambda,\Lambda}(D^2v)\leq F(x, D^2v)\leq \frac{f(x)}{\mathcal{H}(x, \nabla v)}\le \frac{f^{+}(x)}{\mathcal{H}(x, \nabla v)}
\]
whenever $|\nabla v|\geq\mathrm{M}_{\mathrm{F}} = \varepsilon_0>0$ so that
\[
\mathscr{M}^-_{\lambda,\Lambda}(D^2v)-\Lambda|\nabla v|-f_0(x)\leq \left(\frac{1}{\mathcal{H}(x, \nabla v)}-\frac{L_1^{-1}}{\varepsilon_0^p + \mathfrak{a}(x)\varepsilon_0^q}\right)f^{+}(x)\leq 0.
\]

Remember that the constants obtained in \cite{IS2} are monotone with respect to $\tau$ and bounded away from zero and infinity, so we get a uniform estimate as \eqref{eq.lep} for supersolutions of $\mathcal{G}[v]\defeq \mathcal{H}(x, \nabla v)F(x, D^2v)$.

Therefore, in such a situation we have (recall $\sigma(x)\equiv\Lambda$) from Theorem \ref{wharn}:
 \begin{equation}\label{EqWHN}
      \|v\|_{L^{p_0}\left(B_{\frac{1}{4}}\right)}\leq\displaystyle C.\left\{\inf_{B_{\frac{1}{2}}} v + \varepsilon_0 + \left\|f_0\right\|_{L^{N}(B_1)}\right\} \leq \Xi_0,
 \end{equation}
where
$$
\Xi_0 \defeq \left\{
 \begin{array}{lcl}
   \displaystyle C.\left\{\inf_{B_{\frac{1}{2}}} v + \min\left\{1, \left[(q+1)\sqrt[N]{|B_1|}L_1^{-1}\left\|\frac{f^{+}}{1+ \mathfrak{a}}\right\|_{L^{\infty}(B_1)}\right]^{\frac{1}{q+1}}\right\}\right\} & \text{if} & \varepsilon_0 \in (0, 1] \\
   \displaystyle C.\left\{\inf_{B_{\frac{1}{2}}} v + \min\left\{1,\left[(p+1)\sqrt[N]{|B_1|}L_1^{-1}\left\|\frac{f^{+}}{1+ \mathfrak{a}}\right\|_{L^{\infty}(B_1)}\right]^{\frac{1}{p+1}}\right\}\right\} & \text{if} & \varepsilon_0 \in (1, \infty).
 \end{array}
 \right.
$$
Notice that we have used in above inequalities that the function
$$
(0, \infty) \ni t \mapsto \mathfrak{h}(t) = t + \frac{1}{t^s} \left(\sqrt[N]{|B_1|}L_1^{-1}\left\|\frac{f^{+}}{1+ \mathfrak{a}}\right\|_{L^{\infty}(B_1)}\right)
$$
is optimized (i.e. its lowest upper bound) when $t^{\ast} = \left(s\sqrt[N]{|B_1|}L_1^{-1}\left\|\frac{f^{+}}{1+ \mathfrak{a}}\right\|_{L^{\infty}(B_1)}\right)^{\frac{1}{s+1}}$ for $s \in (0, \infty)$.

In conclusion, in any case, we obtain (since $0<p\le q<\infty$)
{\small{
$$
 \displaystyle \|v\|_{L^{p_0}\left(B_{\frac{1}{4}}\right)}\leq  C.\left\{\inf_{B_{\frac{1}{2}}} v + (q+1)^{\frac{1}{q+1}}\Pi^{f^{+}, \mathfrak{a}}_{p, q}\right\}
$$
}}
where
$$
\Pi^{f^+, \mathfrak{a}}_{p, q} \defeq \max\left\{\left[\sqrt[N]{|B_1|}L_1^{-1}\left\|\frac{f^{+}}{1+ \mathfrak{a}}\right\|_{L^{\infty}(B_1)}\right]^{\frac{1}{p+1}}, \left[\sqrt[N]{|B_1|}L_1^{-1}\left\|\frac{f^{+}}{1+ \mathfrak{a}}\right\|_{L^{\infty}(B_1)}\right]^{\frac{1}{q+1}}\right\}.
$$

Similarly, from Theorem \ref{localmax}, if
\[
\mathcal{H}(x, \nabla v)F(x, D^2v)\geq f(x) \quad \text{in} \quad B_1
\]
in the viscosity sense we again have
\[
\mathscr{M}^+_{\lambda,\Lambda}(D^2v)\geq F(x, D^2v) \geq \frac{f(x)}{\mathcal{H}(x, \nabla v)} \geq -\frac{f^{-}(x)}{\mathcal{H}(x, \nabla v)}\,\,\,\text{whenever}\,\,\, \varepsilon_0  = \mathrm{M}_{\mathrm{F}} \leq |\nabla u|,
\]
and we can set, similarly as above, $f_0(x)\defeq \frac{L_1^{-1}f^{-}(x)}{\varepsilon_0^{p}+\mathfrak{a}(x)\varepsilon_0^{p}}
$
to get
\[
\mathscr{M}^+_{\lambda,\Lambda}(D^2v)+\Lambda|\nabla v|+f_0(x)\geq \left(\frac{L_1^{-1}}{\varepsilon_0^{p}+\mathfrak{a}(x)\varepsilon_0^{p}}-\frac{1}{\mathcal{H}(x, \nabla v)}\right)f^{-}(x) \geq 0.
\]

Therefore, in such a setting, we have from Theorem \ref{localmax}:
{\small{
\begin{equation}\label{EqLMP}
\begin{array}{ccl}
 \displaystyle \sup_{B_{\frac{1}{2}}} v & \leq & \displaystyle \|u^+\|_{L^{p_1}(B_1)}+ \varepsilon_0+ \left\|f_0\right\|_{L^{N}(B_1)}\\
   & \le & \Xi_1,
\end{array}
\end{equation}
}}
where, as before, we can estimate
{\scriptsize{
$$
\Xi_1\defeq\left\{
 \begin{array}{lcl}
   \displaystyle C.\left\{\|v^+\|_{L^{p_1}(B_1)} + \min\left\{1,\left[(q+1)\sqrt[N]{|B_1|}L_1^{-1}\left\|\frac{f^{-}}{1+ \mathfrak{a}}\right\|_{L^{\infty}(B_1)}\right]^{\frac{1}{q+1}}\right\}\right\} & \text{if} & \varepsilon_0 \in (0, 1] \\
   \displaystyle C.\left\{\|v^+\|_{L^{p_1}(B_1)} + \min\left\{1, \left[(p+1)\sqrt[N]{|B_1|}L_1^{-1}\left\|\frac{f^{-}}{1+ \mathfrak{a}}\right\|_{L^{\infty}(B_1)}\right]^{\frac{1}{p+1}}\right\}\right\} & \text{if} & \varepsilon_0 \in (1, \infty),
 \end{array}
 \right.
$$
}}
Therefore, in any setting (since $0<p\le q<\infty$)
$$
\displaystyle \sup_{B_{\frac{1}{2}}} v \le C.\left\{\|v^+\|_{L^{p_1}(B_1)} + (q+1)^{\frac{1}{q+1}}\Pi^{f^-, \mathfrak{a}}_{p, q}\right\}
$$
thereby finishing this analysis.

Finally, by combining \eqref{EqWHN} and \eqref{EqLMP}, we obtain the following Harnack inequality for viscosity solutions:

\begin{theorem}[{\bf Harnack inequality}]\label{ThmHarIneq} Let $u$ be a non-negative viscosity solution to
$$
F_0(x, \nabla v, D^2v) = 0 \quad \text{in} \quad B_1.
$$
Then,
$$
\displaystyle \sup_{B_{\frac{1}{2}}} u(x) \leq \mathrm{C} \cdot\left\{\inf_{B_{\frac{1}{2}}} u(x) + (q+1)^{\frac{1}{q+1}}\Pi^{f, \mathfrak{a}}_{p, q}\right\},
$$
where $\mathrm{C}>0$ depends only on $\displaystyle N, \lambda$ and $\Lambda$.
\end{theorem}

\begin{remark}[{\bf Harnack inequality - scaled version}]\label{HarnIneqSV} For our purposes, it will be useful to obtain a scaled version of Harnack inequality. Indeed, let $v$ be a non-negative viscosity solution to
$$
\mathcal{G}(x, \nabla v, D^2v) = f(x) \quad \text{in} \quad B_{r} \quad \text{for a fixed} \quad r \in (0, \infty)
$$
where (A0)-(A2), \eqref{1.2} and \eqref{1.3} are in force. Then,
$$
\displaystyle \sup_{B_{\frac{r}{2}}} v(x) \leq \mathrm{C} \cdot\left\{\inf_{B_{\frac{r}{2}}} v(x) + (q+1)^{\frac{1}{q+1}}\max\left\{r^{\frac{p+2}{p+1}}, r^{\frac{q+2}{p+1}}\right\}\Pi^{f, \mathfrak{a}}_{p, q}\right\},
$$
where $\mathrm{C}(N, \lambda, \Lambda)>0$.
\end{remark}

Finally, from Harnack inequality (Theorem \ref{ThmHarIneq}) and making use of arguments in  \cite[Proposition 4.10]{CC95}, we can obtain, in a standard way, the following interior H\"{o}lder regularity result (cf. \cite[Theorem 2]{DeF20}).

\begin{theorem}[{\bf Local H\"{o}lder estimate}]\label{HoldEstThm} Let $u$ be a viscosity solution to
$$
   F_0(x, \nabla v, D^2v) = 0 \quad \text{in} \quad B_1.
$$
where $f$ is a continuous and bounded function. Then, $u \in C_{\text{loc}}^{0, \alpha}(B_1)$ for some universal $\alpha \in (0, 1)$. Moreover,
$$
   \displaystyle \|u\|_{C^{0, \alpha}\left(\overline{B_{\frac{1}{2}}}\right)}\leq \mathrm{C} \cdot\left\{\|u\|_{L^{\infty}(B_1)} + (q+1)^{\frac{1}{q+1}}\Pi^{f}_{p, q}\right\},
$$
where $\mathrm{C}>0$ depends only on $N, \lambda$ and $\Lambda$.
\end{theorem}

\subsection{An Alexandroff-Bakelman-Pucci type estimate}\label{Append3}

In the sequel, we will deliver an ABP estimate adapted to our context of fully nonlinear models with non-homogeneous degeneracy (cf. \cite[Theorem 1]{DFQ} and \cite[Theorem 1.1]{JunMio10}). Such an estimate is pivotal in order to obtain universal bounds for viscosity solutions in terms of data of the problem.

\begin{theorem}[{\bf Alexandroff-Bakelman-Pucci estimate}] \label{ABPthm}
  Assume that assumptions (A0)-(A2) there hold. Then, there exists $C = C(N, \lambda, p, q, \diam(\Omega))>0$ such that for any $u \in C^0(\overline{\Omega})$ viscosity sub-solution (resp. super-solution) of \eqref{EqPrinc} in $\{x \in \Omega:u(x)>0\}$ (resp. $\{x \in \Omega:u(x)<0\}$, satisfies
$$
\displaystyle \sup_{\Omega} u(x) \leq \sup_{\partial \Omega} u^{+}(x) +C\cdot\diam(\Omega)\max\left\{\left\|\frac{f^{-}}{1+\mathfrak{a}}\right\|^{\frac{1}{p+1}}_{L^N(\Gamma^{+}(u^{+}))}, \left\|\frac{f^{-}}{1+\mathfrak{a}}\right\|^{\frac{1}{q+1}}_{L^N(\Gamma^{+}(u^{+}))}\right\},
$$
{\small{
$$
\left(\text{resp.} \,\,\,\displaystyle \sup_{\Omega} u^{-}(x) \leq \sup_{\partial \Omega} u^{-}(x) +C\cdot\diam(\Omega)\max\left\{\left\|\frac{f^{+}}{1+\mathfrak{a}}\right\|^{\frac{1}{p+1}}_{L^N(\Gamma^{+}(u^{-}))}, \left\|\frac{f^{+}}{1+\mathfrak{a}}\right\|^{\frac{1}{q+1}}_{L^N(\Gamma^{+}(u^{-}))}
\right\}\right)
$$}}
where
$$
\Gamma^{+}(u) \defeq \left\{x \in \Omega: \exists \,\, \xi \in \R^N\,\,\,\text{such that}\,\,u(y)\le u(x)+\langle \xi, y-x\rangle \,\, \forall\, y \in \Omega\right\}.
$$

Particularly, we conclude that
$$
\displaystyle \|u\|_{L^{\infty}(\Omega)} \leq \|u\|_{L^{\infty}(\partial \Omega)} +C\cdot\diam(\Omega)\max\left\{\left\|\frac{f}{1+\mathfrak{a}}\right\|^{\frac{1}{p+1}}_{L^N(\Omega)}, \left\|\frac{f}{1+\mathfrak{a}}\right\|^{\frac{1}{q+1}}_{L^N(\Omega)}\right\}
$$
\end{theorem}

\begin{proof} We are going to prove the first estimate, the second one follows by similar reasoning. In the sequel, we will show that our class of operators fits into the framework of \cite[Theorem 1]{Imb}. For that purpose, as before, let us consider
$$
   F_0(x, \xi, \mathrm{X}) \defeq \mathcal{H}(x, \xi)F(x, \mathrm{X})-f(x) \quad \text{and} \quad
f_0(x)\defeq \frac{L_1^{-1}f^{+}(x)}{\varepsilon_0^p + \mathfrak{a}(x)\varepsilon_0^q} \quad \text{for fixed} \quad \varepsilon_0 \in (0, \infty).
$$

Now, if we have (in the viscosity sense)
\[
\mathcal{H}(x, \nabla u)F(x, D^2u)\leq f(x) \quad \text{in} \quad B_1,
\]
then condition (A1) and by supposing $|\nabla v|\geq\mathrm{M}_{\mathrm{F}} = \varepsilon_0>0$  ensure that
\[
\mathscr{M}^-_{\lambda,\Lambda}(D^2u)\leq F(x, D^2u)\le \frac{f^{+}(x)}{\mathcal{H}(x, \nabla u)}.
\]
Consequently,
\[
\mathscr{M}^-_{\lambda,\Lambda}(D^2u)-\Lambda|\nabla u|-f_0(x)\leq \left(\frac{1}{\mathcal{H}(x, \nabla u)}-\frac{L_1^{-1}}{\varepsilon_0^p + \mathfrak{a}(x)\varepsilon_0^q}\right)f^{+}(x)\leq 0.
\]

Therefore, $u$ is a viscosity super-solution of a uniformly elliptic problem with ``large'' gradient. From ABP estimate in \cite[Theorem 1]{Imb} we obtain that

\begin{equation}
\displaystyle \sup_{\Omega} u^{-}(x) \leq \sup_{\partial \Omega} u^{-}(x) + C\cdot \diam(\Omega)\Big( \epsilon_{0}+\left\|f_0\right\|_{L^N(\Gamma^{+}(u^{-}))}\Big).
\end{equation}

Next, we will separate the analysis in two cases: First, if $\epsilon_{0} \in (0,1]$, then
\begin{equation}\label{EqABP1}
  \begin{array}{rcl}
    \displaystyle \sup_{\Omega} u^{-}(x) & \le & \displaystyle \sup_{\partial \Omega} u^{-}(x) + C\cdot \diam(\Omega)\left( \epsilon_{0}+L^{-1}_{1}\frac{1}{\epsilon_{0}^{q}}\left\|\frac{f^{=}}{1+\mathfrak{a}}\right\|_{L^N(\Gamma^{+}(u^{-}))} \right) \\
     & \le & \displaystyle \sup_{\partial \Omega} u^{-}(x) + C\cdot \diam(\Omega)\min\left\{1, \left( (q+1)L^{-1}_{1}\left\|\frac{f^{=}}{1+\mathfrak{a}}\right\|_{L^N(\Gamma^{+}(u^{-}))} \right)^{\frac{1}{q+1}}\right\}
  \end{array}
\end{equation}

On the other hand, if $\epsilon_{0} \in (1,\infty)$ then

\begin{equation}\label{EqABP2}
 \sup_{\Omega} u^{-}(x) \leq \sup_{\partial \Omega} u^{-}(x) + C\cdot \diam(\Omega)\min\left\{1, \left( (p+1)L^{-1}_{1}\left\|\frac{f^{=}}{1+\mathfrak{a}}\right\|_{L^N(\Gamma^{+}(u^{-}))} \right)^{\frac{1}{p+1}}\right\}.
\end{equation}

Therefore, by combining inequalities \eqref{EqABP1} and \eqref{EqABP2} we conclude that
{\small{
$$
  \displaystyle  \sup_{\Omega} u^{-}(x) \leq \sup_{\partial \Omega} u^{-}(x) + C(\diam(\Omega),p,q,L_{1},\Lambda)\max\left\{\left\|\frac{f^{-}}{1+\mathfrak{a}}\right\|^{\frac{1}{p+1}}_{L^N(\Gamma^{+}(u^{-}))}, \left\|\frac{f^{-}}{1+\mathfrak{a}}\right\|^{\frac{1}{q+1}}_{L^N(\Gamma^{+}(u^{-}))}\right\}.
$$}}
\end{proof}

\subsection{An inhomogeneous Hopf type result}\label{Append4}

In this final part, we will present a pivotal tool in proving uniform Lipschitz estimates of solutions, namely a quantitative version of the Hopf Lemma, in the inhomogeneous setting for fully nonlinear problems with non-homogeneous degeneracy (cf. \cite[Lemma 2.10]{RS2} for the uniformly elliptic and homogeneous case).

\begin{lemma}[{\bf Inhomogeneous Hopf type Lemma}]\label{lemma2.1}
Suppose that the assumptions (A0)-(A1) and \eqref{1.2} are in force. Let $u$ be a positive viscosity solution to
$$
	\mathcal{G}(x, \nabla u, D^2 u) = f(x) \quad \mbox{in} \quad  B_{R}(z_0)
$$
where $f \in L^{\infty}(B_{R}(z_0))$. Assume further that for some $x_0 \in \partial B_R(z_0)$,
$$
	u(x_0)=0 \quad \textrm{and} \quad \frac{\partial u}{\partial \nu}(x_0) \le \Im,
$$
where $\nu$ is the inward normal direction at $x_0$. Then, for any $r \in (0, 1)$, there exists a constant $\mathrm{C}_0(\verb"universal") >0$ such that
$$
\displaystyle \sup_{B_{\frac{rR}{2}}(z_0)} u(x) \leq \mathrm{C}_0R.\left\{\Im + \max\left\{(r^{p+2}R)^{\frac{1}{p+1}}, (r^{p+2}R)^{\frac{1}{q+1}}\right\}\Pi^{f, r^{p-q}\mathfrak{a}(z_0 + rRx)}_{p, q}\right\}
$$
\end{lemma}

\begin{proof}
First, it is sufficient to consider the scaled function $v_{z_0, R}: B_{1} \to \mathbb{R}$ given by
 $$
    v_{z_0, R}(x) \defeq \frac{u(z_0+rRx)}{R},
 $$
for $r \in (0, 1)$ to be determined \textit{a posteriori}.

In effect, $v_{z_0, R}$ is a non-negative viscosity solution of
$$
     \mathcal{H}_{z_0, R}(y,\nabla v_{z_0, R})F_{z_0, R}(x, D^{2}v_{z_0, R})=f_{z_0, R}(x) \quad \text{in} \quad B_1
$$
where
$$
\left\{
\begin{array}{rcl}
  F_{z_0, R}(x, \mathrm{X}) & \defeq & r^2RF\left(z_0+rR x,\frac{1}{r^2R}\mathrm{X}\right) \\
  \mathcal{H}_{z_0, R}(x, \xi) & \defeq &  r^p\mathcal{H}\left(z_0+rRx, \frac{1}{r}\xi\right)\\
  \mathfrak{a}_{z_0, R}(x) & \defeq & r^{p-q}\mathfrak{a}(z_0+rRx)\\
  f_{z_0, R}(x) & \defeq & r^{p+2}Rf(z + rRx)
\end{array}
\right.
$$
Moreover, $F_{z_0, R}, \mathcal{H}_{z_0, R}$ and $\mathfrak{a}_{z_0, R}$ fulfil the structural assumptions (A0)-(A2), \eqref{1.2} and \eqref{1.3}.

Now, let $\mathcal{A}_{\frac{1}{2}, 1}\defeq B_{1}\setminus B_{\frac{1}{2}}$ and define $\Phi:\overline{\mathcal{A}_{\frac{1}{2}, 1}} \to \mathbb{R}_{+}$ a barrier function given by
$$
    \Phi(x)=\mu_0 \cdot \left(e^{-\delta|x|^{2}} - e^{-\delta}\right)
$$
where $\mu_0, \delta>0$ will be chosen \textit{a posteriori}. The gradient and the Hessian of $\Phi$ in $\mathcal{A}_{\frac{1}{2}, 1}$ are

$$    \nabla \Phi(x)=-2\mu_0\delta xe^{-\delta|x|^{2}} \quad \text{and} \quad   D^{2}\Phi(x)=2\mu_0\delta e^{-\delta |x|^{2}}\left(2\delta x\otimes x  - \text{Id}_{\mathrm{N}}\right).
$$

Next, we will show that such a barrier is a viscosity solution to
\begin{equation}
     \mathcal{H}(x,\nabla \Phi)F(x, D^{2}\Phi) > f_{z_0, R}(x) \quad \text{in} \quad \mathcal{A}_{\frac{1}{2}, 1}
 \end{equation}
provided we may adjust appropriately the values of $\mu_0, \delta>0$ and $r>0$.

For this purpose, notice that for $\delta > \frac{\Lambda(N-1)}{2\lambda}$, then $\Phi$ is a convex and decreasing function in the annular region $\mathcal{A}_{\frac{1}{2}, 1}$. Hence,
$$
\mathscr{M}^{-}_{\lambda,\Lambda}(D^2 \Phi(x)) = 2\mu_0\delta e^{-\delta|x|^{2}}\left[2\delta \lambda - \Lambda(N-1)\right] \geq 2\mu_0\delta e^{-\delta}\left[2\delta \lambda - \Lambda(N-1)\right] \quad \text{in} \quad \mathcal{A}_{\frac{1}{2}, 1}.
$$
In the sequel, we obtain by using the above sentence, (A1) and \eqref{1.2}
$$
\begin{array}{lcl}
  \mathcal{H}(x,\nabla \Phi)F(x, D^{2}\Phi) & \ge &  \mathcal{H}(x,\nabla \Phi)\mathscr{M}^{-}_{\lambda,\Lambda}(D^2 \Phi(x))\\
   & \ge & 2\mu_0\delta e^{-\delta}\left[2\delta \lambda - \Lambda(N-1)\right]\left(\delta^p\mu_0^pe^{-\delta p}+ \delta^q\mu_0^qe^{-\delta q}\mathfrak{a}(x)\right) \\
   & > & r^{p+2}R\|f\|_{L^{\infty}\left(\mathcal{A}_{\frac{rR}{2}, rR}\right)},
\end{array}
$$
which holds true provided we choice $r\ll 1$ small enough by fixing $\mu_0$ and choosing $\delta>0$.

In effect, by choosing $\displaystyle \mu_0 \defeq (e^{-\delta/4}-e^{-\delta})^{-1}\cdot \inf_{\partial B_{\frac{1}{2}}} v_{z_0, R}(x)>0$ follows that
$$
     \Phi(x) \leq v_{z_0, R}(x) \quad  \text{on} \quad  \partial \mathcal{A}_{\frac{1}{2}, 1}.
$$
Thus, by the Comparison Principle (see, Lemma \ref{comparison principle})
 \begin{equation}\label{CompaBarV}
\Phi(x) \leq v_{z_0, R}(x) \quad  \text{in} \quad  \mathcal{A}_{\frac{1}{2}, 1}
 \end{equation}

Therefore, if we label $y_{0} \defeq \frac{x_{0}-z_0}{rR}$, and taking into account \eqref{CompaBarV} and the hypotheses $u(x_{0})=0$, we obtain concerning the normal derivatives in the direction $\nu$ at $x_{0}$ the following

 \begin{equation}
     \mu_0\delta e^{-\delta} \leq \frac{\partial \Phi (y_{0})}{\partial \nu} \leq \frac{\partial v_{z_0, R} (y_{0})}{\partial \nu} \leq \Im.
 \end{equation}
Thus,
$$
\inf_{\partial B_{\frac{1}{2}}} v_{z_0, R}(x)\leq \Im\delta^{-1}\cdot \left(e^{-\frac{3}{4}\delta}-1\right)
$$

On the other hand, by the Harnack inequality (see, Theorem \ref{ThmHarIneq}) we have that
$$
\begin{array}{lcl}
  \displaystyle \sup_{B_{\frac{1}{2}}} v_{z_0, R}(x) & \leq & \displaystyle \mathrm{C}\cdot\left\{\inf_{B_{\frac{1}{2}}} v_{z_0, R}  + (q+1)^{\frac{1}{q+1}}\max\left\{(r^{p+2}R)^{\frac{1}{p+1}}, (r^{p+2}R)^{\frac{1}{q+1}}\right\}\Pi^{f, \mathfrak{a}_{z_0, R}}_{p, q}\right\} \\
   & \le &\displaystyle \mathrm{C}\cdot \left\{\inf_{\partial B_{\frac{1}{2}}} v_{z_0, R}  + (q+1)^{\frac{1}{q+1}}\max\left\{(r^{p+2}R)^{\frac{1}{p+1}}, (r^{p+2}R)^{\frac{1}{q+1}}\right\}\Pi^{f, \mathfrak{a}_{z_0, R}}_{p, q}\right\} \\
   & \le & \displaystyle \mathrm{C}_0\cdot \left\{\Im + \max\left\{(r^{p+2}R)^{\frac{1}{p+1}}, (r^{p+2}R)^{\frac{1}{q+1}}\right\}\Pi^{f, \mathfrak{a}_{z_0, R}}_{p, q}\right\}.
\end{array}
$$
and by using the definition of $v_{z_0, R}$ we conclude that
$$
\displaystyle \sup_{B_{\frac{rR}{2}}(z_0)} u(x) \leq \mathrm{C}_0R\cdot\left\{\Im + \max\left\{(r^{p+2}R)^{\frac{1}{p+1}}, (r^{p+2}R)^{\frac{1}{q+1}}\right\}\Pi^{f, \mathfrak{a}_{z_0, R}}_{p, q}\right\}
$$
\end{proof}

\bigskip

\noindent{\bf Acknowledgements.} This manuscript is part of the second author's Ph.D thesis. He would like to thank the Department of Mathematics at Universidade Federal do Cear\'{a} for fostering a pleasant and productive scientific atmosphere, which has benefited a lot the final outcome of this project. E.C. J\'{u}nior thanks to Capes-Brazil (Doctoral Scholarship). J.V. da Silva and G.C. Ricarte have been partially supported by CNPq-Brazil under Grants No. 310303/2019-2 and No. 303078/2018-9.

\bigskip

\noindent \textsc{Jo\~{a}o Vitor da Silva } \hfill \textsc{Elzon C. J\'{u}nior}\\
Universidade Estadual de Campinas \hfill  Universidade Federal Cear\'a \\
Departmento de Matem\'{a}tica - IMECC \hfill Departmento de Matem\'{a}tica \\
Campinas - SP - Brazil 13083-859  \hfill Fortaleza, CE-Brazil 60455-760\\
\texttt{jdasilva@unicamp.br} \hfill \texttt{bezerraelzon@gmail.com}

\vspace{1cm}
\noindent \textsc{Gleydson C. Ricarte}\\
Universidade Federal Cear\'a \\
Departmento de Matem\'{a}tica \\
Fortaleza, CE-Brazil 60455-760\\
\texttt{ricarte@mat.ufc.br}

\begin{thebibliography}{99}

\bibitem{AC81} H.W. Alt and L.A. Caffarelli,
\textit{Existence and regularity for a minimum problem with free boundary}.
J. Reine Angew. Math. 325 (1981), 105-144.

\bibitem{ART15} D.J. Ara\'{u}jo, G.C. Ricarte, E.V. Teixeira,
\textit{Geometric gradient estimates for solutions to degenerate elliptic equations}.
Calc. Var. Partial Differential Equations 53 (2015), 605-625.

\bibitem{ART17} D.J. Ara\'{u}jo, G.C. Ricarte, E.V. Teixeira,
\textit{Singularly perturbed equations of degenerate type}.
Ann. Inst. H. Poincar\'{e} Anal. Non Lin\'{e}aire 34 (2017), no. 3, 655-678.

\bibitem{BCM15III} P. Baroni, M. Colombo and G. Mingione,
\textit{Regularity for general functionals with double phase}.
Calc. Var. Partial Differential Equations 57 (2018), no. 2, Art. 62, 48 pp.

\bibitem{BCN} H. Berestycki, L.A. Caffarelli and L. Nirenberg,
\textit{Uniform estimates for regularization of free boundary problems.}
Analysis and partial differential equations, 567–619, Lecture Notes in Pure and Appl. Math., 122, Dekker, New York, 1990.

\bibitem{BD04} I. Birindelli and F. Demengel,
\textit{Comparison principle and Liouville type results for singular fully nonlinear operators}.
Ann. Fac. Sci. Toulouse Math. (6) 13 (2004), no. 2, 261-287.

\bibitem{BD07} I. Birindelli and F. Demengel,
\textit{Eigenvalue, maximum principle and regularity for fully non linear homogeneous operators}.
Commun. Pure Appl. Anal. 6 (2007), no. 2, 335–366.

\bibitem{BD07-2} I. Birindelli and F. Demengel,
\textit{The Dirichlet problem for singular fully nonlinear operators}.
Discrete Contin. Dyn. Syst. (2007), 110-121. Special vol.

\bibitem{BD1} I. Birindelli and F. Demengel,
\textit{Regularity for radial solutions of degenerate fully nonlinear equations}.
Nonlinear Anal. 75 (2012), no. 17, 6237-6249.

\bibitem{BD2} I. Birindelli and F. Demengel,
\textit{$C^{1, \beta}$ regularity for Dirichlet problems associated to fully nonlinear degenerate elliptic equations}.
ESAIM Control Optim. Calc. Var. 20 (2014), no. 4, 1009-1024.

\bibitem{BD3} I. Birindelli and F. Demengel,
\textit{H\"{o}lder regularity of the gradient for solutions of fully nonlinear equations with sub linear first order term}.
Geometric methods in PDE's, 257-268, Springer INdAM Ser., 13, Springer, Cham, 2015.

\bibitem{BDL19} I. Birindelli, F. Demengel and F. Leoni,
\textit{$C^{1, \gamma}$ regularity for singular or degenerate fully nonlinear equations and applications}.
NoDEA Nonlinear Differential Equations Appl. 26 (2019), no. 5, Paper No. 40, 13 pp.

\bibitem{C1} L.A. Caffarelli,
\textit{Interior a priori estimates for solutions of fully nonlinear equations.}
Ann. of Math. (2) \textbf{130} (1989), no. 1, 189--213.

\bibitem{CC95} L.A. Caffarelli and X. Cabr\'{e},
\textit{Fully nonlinear elliptic equations}.
American Mathematical Society Colloquium Publications, 43. American Mathematical Society, Providence, RI, 1995. vi+104 pp. ISBN: 0-8218-0437-5.

\bibitem{CK-AM04} L.A. Caffarelli, K-A. Lee and A. Mellet,
\textit{Singular limit and homogenization for flame propagation in periodic excitable media}.
Arch. Ration. Mech. Anal. 172 (2004), no. 2, 153–190.

\bibitem{CLW1} L.A. Caffarelli, C. Lederman and N. Wolanski,
\textit{Uniform estimates and limits for a two phase parabolic singular perturbation problem}
Indiana Univ. Math. J. 46 (1997), no. 2, 453–489.

\bibitem{CLW2} L.A. Caffarelli, C. Lederman and N. Wolanski,
\textit{Pointwise and viscosity solutions for the limit of a two phase parabolic singular perturbation problem}.
Indiana Univ. Math. J. 46 (1997), no. 3, 719–740.

\bibitem{CS05} L. Caffarelli and S. Salsa,
\textit{A geometric approach to free boundary problems}.
Graduate Studies in Mathematics, 68. American Mathematical Society, Providence, RI, 2005. x+270 pp. ISBN: 0-8218-3784-2.

\bibitem{UserG} Crandall, M.; Ishii, H.; Lions, P-L.
\textit{User's guide to viscosity solutions of second order partial differential equations.}
Bull. Amer. Math. Soc. (N.S.) {\bf 27} (1992), no. 1, 1--67.

\bibitem{DPS03} D. Danielli, A. Petrosyan and H. Shahgholian,
\textit{A singular perturbation problem for the $p$-Laplace operator},
Indiana Univ. Math. J. 52 (2003), 457-476.

\bibitem{daSLR20} J.V. da Silva, R.A. Leit\~{a}o J\'{u}nior and  G.C. Ricarte,
\textit{Geometric regularity estimates for fully nonlinear elliptic equations with free boundaries}.
Mathematische Nachrichten, Vol. 294(1) 2021 p. 38-55. DOI:\url{10.1002/mana.201800555}.

\bibitem{daSRRV21} J.V. da Silva, G.C. Rampasso, G.C. Ricarte and H. Vivas,
\textit{Free boundary regularity for a class of one-phase problems with non-homogeneous degeneracy}. Arxiv Preprint \url{arXiv:2103.11028}.

\bibitem{daSR19} J.V. da Silva and G.C. Ricarte,
\textit{An asymptotic treatment for non-convex fully non-linear elliptic equations: Global Sobolev and BMO type estimates}.
 Comm. Contemp. Math. 21, no. 7, 1850053, 28 pp (2019)

\bibitem{daSR20} J.V. da Silva and G.C. Ricarte,
\textit{Geometric gradient estimates for fully nonlinear models with non-homogeneous degeneracy and applications}.
Calc. Var. Partial Differential Equations  Calc. Var. 59, 161 (2020).

\bibitem{daSRS19I} J.V. da Silva, J.D. Rossi and A. Salort,
\textit{Regularity properties for $p-$dead core problems and their asymptotic limit as $p \to \infty$},
J. Lond. Math. Soc. (2) 99 (2019), no. 1, 69-96.

\bibitem{daSS18} J.V. da Silva and A. Salort,
\textit{Sharp regularity estimates for quasi-linear elliptic dead core problems and applications}.
Calc. Var. Partial Differential Equations 57 (2018), no. 3, 57: 83.

\bibitem{DSVI} J.V. da Silva and H. Vivas,
\textit{ The obstacle problem for a class of degenerate fully nonlinear operators},
To appear in Revista Matem\'{a}tica Iberoamericana, doi \url{10.4171/rmi/1256}.

\bibitem{DSVII} J.V. da Silva and H. Vivas,
\textit{ Sharp regularity for degenerate obstacle type problems: a geometric approach},
Discrete $\&$ Continuous Dynamical Systems - A, , 2021, 41 (3) : 1359-1385.

\bibitem{DFQ} G. D\'avila, P. Felmer and  A. Quaas,
\textit{ Alexandroff-Bakelman-Pucci estimate for singular or degenerate fully nonlinear elliptic equations.}
C. R. Math. Acad. Sci. Paris 347 (2009), no. 19-20, 1165--1168.

\bibitem{DeF20} C. De Filippis,
\textit{Regularity for solutions of fully nonlinear elliptic equations with nonhomogeneous degeneracy}.
Proc. Roy. Soc. Edinburgh Sect. A 151 (2021), no. 1, 110-132.

\bibitem{DeFM19I} C. De Filippis and G. Mingione,
\textit{Manifold Constrained Non-uniformly Elliptic Problems}.
The Journal of Geometric Analysis, 30 (2) (2020) 1661-1723.

\bibitem{DeFM19II} C. De Filippis and G. Mingione,
\textit{On the Regularity of Minima of Non-autonomous functionals}.
The Journal of Geometric Analysis, 30 (2) (2020) 1584-1626.

\bibitem{DeFM20} C. De Filippis and G. Mingione,
\textit{Lipschitz bounds and nonautonomous integrals}.
Arxiv Preprint (2020). \url{arXiv:2007.07469}.

\bibitem{DeFO} C. De Filippis and J. Oh,
\textit{Regularity for multi-phase variational problems}.
J. Differential Equations 267 (2019), no. 3, 1631–1670.

\bibitem{DeSFS15} D. De Silva, F. Ferrari and S. Salsa,
\textit{Free boundary regularity for fully nonlinear non-homogeneous two-phase problems}.
J. Math. Pures Appl. (9) 103 (2015), no. 3, 658–694.

\bibitem{dosPT16} D. dos Prazeres and E.V. Teixeira,
\textit{Cavity problems in discontinuous media}.
Calc. Var. Partial Differential Equations 55 (2016), no. 1, Art. 10, 15 pp.

\bibitem{Ev82} L.C. Evans,
\textit{Classical solutions of fully nonlinear, convex, second-order elliptic equations}.
Comm. Pure Appl. Math., 35(3). 333-363, 1982.

\bibitem{EG92} L.C. Evans and R.F. Gariepy,
\textit{Measure theory and fine properties of functions}.
Studies in Advanced Mathematics. CRC Press, Boca Raton, FL, 1992. viii+268 pp. ISBN: 0-8493-7157-0.

\bibitem{H13} M.H. Holmes,
\textit{Introduction to Perturbation Methods}.
Second edition. Texts in Applied Mathematics, 20. Springer, New York, 2013. xviii+436 pp. ISBN: 978-1-4614-5476-2; 978-1-4614-5477-9 34-01

\bibitem{Imb} C. Imbert,
\textit{ Alexandroff-Bakelman-Pucci estimate and Harnack inequality for degenerate/singular fully non-linear elliptic equations.}
J. Differential Equations {\ f 250} (2011), no. 3, 1553--1574.

\bibitem{IS} C. Imbert and L. Silvestre,
\textit{ $C^{1,\alpha}$ regularity of solutions of some degenerate fully non-linear elliptic equations.}
Adv. Math. 233 (2013), 196--206.

\bibitem{IS2} C. Imbert and L. Silvestre,
\textit{Estimates on elliptic equations that hold only where the gradient is large}.
J. Eur. Math. Soc. (JEMS) 18 (2016), no. 6, 1321–1338.

\bibitem{JunMio10} T. Junges Miotto,
\textit{The Aleksandrov-Bakelman-Pucci estimates for singular fully nonlinear operators}.
Commun. Contemp. Math. 12 (2010), no. 4, 607-627.

\bibitem{Kar19} A. Karakhanyan,
\textit{Regularity for the two phase singular perturbation problems}.
To appear in  Proceedings of the London Mathematical Society Arxiv Preprint \url{arXiv:1910.06997}.

\bibitem{Kry82} N.V. Krylov,
\textit{Boundedly inhomogeneous elliptic and parabolic equations}.
Izv. Akad. Nauk SSSR Ser. Mat. 46 (1982), no. 3, 487–523, 670.

\bibitem{Kry83} N.V. Krylov,
\textit{Boundedly inhomogeneous elliptic and parabolic equations in a domain}.
Izv. Akad. Nauk SSSR Ser. Mat., 47(1): 75-108, 1983.

\bibitem{LVW01} C. Lederman, J.L. V\'{a}zquez and N. Wolanski,
\textit{Uniqueness of solution to a free boundary problem from combustion}.
Trans. Amer. Math. Soc. 353 (2001), no. 2, 655–692

\bibitem{LW98} C. Lederman and N. Wolanski,
\textit{Viscosity solutions and regularity of the free boundary for the limit of an elliptic two phase singular perturbation problem}.
Ann. Scuola Norm. Sup. Pisa Cl. Sci. (4) 27 (1998), no. 2, 253–288 (1999).

\bibitem{LW06} C. Lederman and N. Wolanski,
\textit{A two phase elliptic singular perturbation problem with a forcing term}.
J. Math. Pures Appl. (9) 86 (2006), no. 6, 552–589.

\bibitem{LW08} C. Lederman and N. Wolanski,
\textit{A local monotonicity formula for an inhomogeneous singular perturbation problem and applications}.
Ann. Mat. Pura Appl. (4) 187 (2008), no. 2, 197–220.

\bibitem{LW10} C. Lederman and N. Wolanski,
\textit{A local monotonicity formula for an inhomogeneous singular perturbation problem and applications. II.}
Ann. Mat. Pura Appl. (4) 189 (2010), no. 1, 25–46.

\bibitem{LW16} C. Lederman and N. Wolanski,
\textit{An inhomogeneous singular perturbation problem for the $p(x)-$Laplacian}.
Nonlinear Anal. 138 (2016), 300-325.

\bibitem{LD18} W. Liu, and G. Dai,
\textit{Existence and multiplicity results for double phase problem}.
J. Differential Equations 265 (2018), no. 9, 4311–4334.

\bibitem{MW09} S. Mart\'{i}nez and N. Wolanski,
\textit{A singular perturbation problem for a quasi-linear operator satisfying the natural growth condition of Lieberman}.
SIAM J. Math. Anal. 41 (2009), no. 1, 318-359.

\bibitem{MorTei07} D. Moreira and E.V. Teixeira,
\textit{A singular perturbation free boundary problem for elliptic equations in divergence form}.
Calc. Var. Partial Differential Equations 29 (2007), no. 2, 161-190.

\bibitem{MoWan1} D. Moreira and L. Wang,
\textit{Singular perturbation method for inhomogeneous nonlinear free boundary problems},
Calc. Var. Partial Differential Equations 49 (2014), 1237-1261.

\bibitem{RS2} G.C. Ricarte and J.V. Silva,
\textit{Regularity up to the boundary for singularly perturbed fully nonlinear elliptic equations},
Interfaces and Free Bound. 17 (2015), 317-332.

\bibitem{RTS17} G.C. Ricarte,, J.V. Silva and R. Teymurazyan,
\textit{Cavity type problems ruled by infinity Laplacian operator}.
J. Differential Equations 262 (2017), no. 3, 2135–2157.

\bibitem{RT11} G.C. Ricarte and  E.V. Teixeira,
\textit{Fully nonlinear singularly perturbed equations and asymptotic free boundaries}.
J. Funct. Anal. 261 (2011), no. 6, 1624--1673

\bibitem{RTU19} G. Ricarte, R. Teymurazyan and J.M. Urbano,
\textit{Singularly perturbed fully nonlinear parabolic problems and their asymptotic free boundaries}.
Rev. Mat. Iberoam. 35 (2019), no. 5, 1535-1558.

\bibitem{ST15} L. Silvestre and E. Teixeira,
\textit{Regularity estimates for fully non linear elliptic equations which are asymptotically convex},
     in \emph{Contributions to nonlinear elliptic equations and systems}, 425-438, Progr. Nonlinear Differential Equations Appl., 86, Birkh\"{a}user/Springer, Cham, 2015.

\bibitem{T0} E.V. Teixeira,
\textit{Optimal regularity of viscosity solutions of fully nonlinear singular equations and their limiting free boundary problems}.
XIV School on Differential Geometry (Portuguese). Mat. Contemp. 30 (2006), 217–237.

\bibitem{Tei-Book} E.V. Teixeira,
\textit{Elliptic regularity and free boundary problems: an introduction}.
Publica\c{c}\~{o}es Matem\'{a}ticas do IMPA. [IMPA Mathematical Publications] 26o Col\'{o}quio Brasileiro de Matem\'{a}tica. [26th Brazilian Mathematics Colloquium] Instituto Nacional de Matem\'{a}tica Pura e Aplicada (IMPA), Rio de Janeiro, 2007. ii+205 pp. ISBN: 978-85-244-0252-4.

\bibitem{T1} E.V. Teixeira,
\textit{A variational treatment for elliptic equations of the flame propagation type: regularity of the free boundary.}
Ann. Inst. H. Poincar\'{e} Anal. Non Lin\'{e}aire, 25 (2008), 633--658.

\bibitem{Tru83} N.S. Trudinger,
\textit{Fully nonlinear, uniformly elliptic equations under natural structure conditions}.
Trans. Amer. Math. Soc. 278 (1983), no. 2, 751-769.

\bibitem{Tru84} N.S. Trudinger,
\textit{Regularity of solutions of fully nonlinear elliptic equations}.
Boll. Un. Mat. Ital. A (6) 3 (1984), no. 3, 421–430.


\bibitem{Weiss03} G.S. Weiss,
\textit{A singular limit arising in combustion theory: fine properties of the free boundary}.
Calc. Var. Partial Differential Equations 17 (2003), no. 3, 311–340.

\bibitem{Z88} L. Zajicek,
\textit{Porosity and $\sigma$-porosity},
Real Anal. Exchange 13 (1987/88), 314-350.

\end{thebibliography}
\end{document}